\documentclass[11pt,a4paper,twoside]{article}
\usepackage[english]{babel}
\usepackage[applemac]{inputenc}
\usepackage[T1]{fontenc}
\usepackage{lmodern}
\usepackage{xcolor}
\usepackage{amsmath}
\usepackage{amsfonts}
\usepackage{amssymb}
\usepackage{amsthm}
\usepackage{fancyhdr}
\usepackage{mathtools}
\usepackage{hyperref}
\usepackage[symbol]{footmisc}

\DeclarePairedDelimiter\abs{\lvert}{\rvert}%
\DeclarePairedDelimiter\norm{\lVert}{\lVert}%
\DeclareMathOperator\arccosh{arccosh}

\theoremstyle{plain}
\newtheorem{theorem}{Theorem}
\newtheorem{proposition}{Proposition}
\newtheorem{corollary}{Corollary}
\newtheorem{lemma}[theorem]{Lemma}
\theoremstyle{definition}
\newtheorem{example}{Example}
\theoremstyle{remark}
\newtheorem{remark}{Remark}

\newcommand{\titlesize}{\fontsize{15}{20pt}\bfseries}
\newcommand{\titulo}[1]{\vspace*{10pt}\begin{center}
               \titlesize{#1}\par\vspace*{10pt}\normalfont}
\newcommand{\autor}[1]{\vspace*{11pt}\sc{#1}\par}
\def\affilnum#1{${}^{#1}$}
\def\affil#1{${}^{#1}$}
\newcommand{\direccion}[1]{\vspace*{12pt}\footnotesize
                \textit{#1}\end{center}\par\vspace*{20pt}}
\begin{document}
 
\thispagestyle{plain}

\titulo 
{A new insight on positivity and contractivity of the Crank-Nicolson scheme for the 
 heat equation\footnote[1]{Research project PID2019-109045GB-C31 funded by Agencia Estatal de Investigaci\'on, Ministerio de Ciencia e Innovaci\'on, Spain.}}

\autor{I. Higueras\affil{1},
\underline{T. Rold\'an}\affil{1}}

\direccion{\affilnum{1}Institute for Advanced Materials and Mathematics\\
Public University of Navarre\\
E-mails: {\tt higueras@unavarra.es, teo@unavarra.es}.}

\begin{abstract}
In this paper we study numerical positivity and contractivity in the infinite norm of Crank-Nicolson method when it is applied to the diffusion equation with homogeneous Dirichlet boundary conditions.  For this purpose, the amplification matrices are written in terms of three kinds of Chebyshev-like polynomials, and necessary and sufficient bounds to  preserve the desired qualitative properties are obtained. For each spatial mesh, we provide the equations that must be solved as well as the intervals that contain these bounds; consequently, they can be easily obtained by a bisection process. Besides, differences between numerical positivity and contractivity are highlighted. This problem has also been addressed by some other authors in the literature and some known results are recovered in our study. Our approach gives a new insight on the problem that completes the panorama and that can be used to study qualitative properties for other problems.
 \end{abstract}


\section{Introduction}
We consider the numerical solution of the one-dimensional heat equation
\begin{subequations}\label{heat_eq}
\renewcommand{\theequation}{\theparentequation.\arabic{equation}}
\begin{align}
\frac{\partial u(t,x)}{\partial t}&= d\, \frac{\partial^2 u(t,x)}{\partial x^2}\,, \quad x\in[0,1], \, t\geq 0\,, \label{heat_eqEDP}
\\[0.5ex]
u(0,t)&=u(1,t)=0\, ,  \quad  t\geq 0, 
\label{heat_cond}
\\[0.5ex]
u(x,0)& =u_0(x)\, , \quad x\in[0,1]\, .  \label{heat_IC}
\end{align}
\end{subequations}
where $u_0$ in \eqref{heat_IC} is a given 
function on $[0,1]$. 
For the sake of simplicity, we consider the diffusion term $d=1$ in the rest of the paper.  
Solutions $u(x,t)$ for the linear parabolic problem \eqref{heat_eqEDP}-\eqref{heat_IC} have several qualitative properties that are relevant in the context of the physical model. In particular, the problem is positivity preserving, that is,  for   $t\geq 0$,
\begin{equation}\label{edp_pos}
u_0(x)\geq 0 \quad \Rightarrow \quad u(x, t)\geq 0\,  ,
\end{equation}
and  the solutions are  monotonically decreasing in time, i.e.,  for $t_2\geq t_1\geq 0$, 
\begin{equation} \label{edp_cont}
\max_{0\leq x \leq 1}|u(t_2,x)|\leq \max_{0\leq x \leq 1}|u(t_1,x)|\, .
 \end{equation}
 In order to obtain numerical approximations with physical sense, properties \eqref{edp_pos}-\eqref{edp_cont} should be preserved in the discretization process. In this paper, we consider the Crank-Nicolson (CN) method, a method of lines approach where second order central finite differences in space are followed by a second order time-stepping method.  
Spatial discretization  of  \eqref{heat_eqEDP}-\eqref{heat_IC}  with second order  central finite differences  and mesh width $h=1/(m+1)$, gives a semi-discrete linear differential system  of the form 
\begin{equation} \label{EDO_heat2}
w'(t)= B_h w(t)\, , \qquad w(0)=w_0\, , \,  t\geq 0\, , 
\end{equation}
where $B_h$ is a matrix of dimension $m$, that is positivity preserving with monotonically decreasing (in the maximum norm)  solutions  (see section \ref{sec:CN} for details). 
Next, a time stepping method is used to obtain numerical approximations $w_n\approx w(t_n)$, where $t_n=n \tau$, and $\tau$ is the constant time stepsize used. In this paper, we consider approximations of the form
\begin{equation}\label{timeiteration}
w_n= A_m \, w_{n-1}\, , \, n\geq 0\, , 
\end{equation}
where $w_0$ is a known value, and $A_m$ is a matrix of dimension $m$ that depends on the time stepping method. In particular, for Runge-Kutta methods,  $A_m=\phi(\tau B_h)$, where $\phi$ is the stability function of the scheme. 

There is a vast list of references   in the analysis of positivity and monotonicity decreasing time stepping schemes  (see, e.g., \cite{shu1988total,thomee1990finite,kraaijevanger1991crk, kraaijevanger1992maximum,horvath1999maximum,ferracina2004stepsize,higueras2004strong,horvath2005positivity,farago2010non,gottlieb2011strong,nusslein2021positivity}; see too \cite{thomee1990finite,gottlieb2011strong,hundsdorfer2003nst} and the references therein). Depending on the context, monotonicity decreasing methods are also known as contractive, SSP (Strong Stability Preserving) or TVD (Total Variation Diminishing) schemes 
(see, e.g., \cite{spijker1985stepsize,shu1988total,kraaijevanger1991crk,higueras2004strong,ferracina2004stepsize,gottlieb2011strong}).  In this setting, stepsize restrictions of the form
\begin{equation}\label{ssp}
\tau\leq C \, \tau_{_{FE}}\, , 
\end{equation}
are obtained,
where $C$ denotes the 
monotonicity threshold factor (also known as SSP-coefficient, radius of absolute monotonicity, contractivity radius,   \ldots)
 of the time stepping method  {(see e.g.,\cite{kraaijevanger1991contractivity,van1986absolute,gottlieb2003strong,spijker1985stepsize,
ferracina2004stepsize})}, and $\tau_{_{FE}}$ is the stepsize restriction for the given qualitative property when forward Euler scheme is used to solve the  specific ODE problem  {(see, e.g. \cite[p. 379]{spijker1985stepsize}, \cite[p. 201]{higueras2004strong}, \cite[pp. 52-53]{gottlieb2011strong})}. In particular, for problem  \eqref{EDO_heat2}, $\tau_{_{FE}}= h^2/2$ for both positivity and contractivity  (see, e.g., \cite[p. 22]{thomee1990finite}), and $C=1$ for forward Euler method    (see section \ref{sec:CN} for details). Observe that with this approach, the  stepsize restriction \eqref{ssp} is the same for all linear systems with the same $\tau_{_{FE}}$  and for some problems this is not a  sharp bound.

A well known time stepping method is the  
$\theta$-method, defined as 
\begin{equation}\label{thetamethod}
w_{n}=(I-\theta\tau B_h)^{-1}(I+(1-\theta)\tau B_h)w_{n-1}\, ,\qquad \theta\in[0,1]\,.
\end{equation} 
Observe that this method 
can be understood as the composition of a  
$(1-\theta)\tau$-step with forward Euler method and a 
$\theta\tau$-step with backward Euler scheme. In particular, for    $\theta=1/2$, CN scheme is obtained.
The  radius of absolute monotonicity for the 
$\theta$-method applied to linear problems is $C_{\theta}=1/(1-\theta)$  \cite{higueras2019strong}. Consequently, numerical positivity and contractivity can be ensured under the restriction (see \cite[p. 135]{thomee1990finite})
\[
\frac{\tau}{h^2}\le \frac{1}{2(1-\theta)}\,.
\]
However, If we  look closer at the iteration matrix $A_m$ in \eqref{thetamethod}, a sharper  bound is possible. Different authors have studied numerical preservation of positivity and monotonicity for  problem  \eqref{heat_eqEDP}-\eqref{heat_IC} with the $\theta$-method
(see, e.g.,  \cite{kraaijevanger1992maximum,horvath1999maximum,farago2010non}).  In 
\cite[p. 72]{farago2010non}, 
the authors obtain  that the numerical solution is positive if and only if 
\begin{equation}\label{positive_inf}
\frac{\tau}{h^2}\le \frac{1-\sqrt{1-\theta}}{\theta(1-\theta)} \, , 
\end{equation}
whereas in 
\cite[Remark 7.1]{kraaijevanger1992maximum} and \cite[p. 456]{horvath1999maximum} it is shown that the numerical solution is contractive if and only if
\begin{equation}\label{contractive_inf}
\frac{\tau}{h^2}\le \frac{2-\theta}{4(1-\theta)^2} \, .
\end{equation}
On the following we will denote 
$s=\tau/h^2$ to the CFL coefficient.
The contractivity result in \cite{kraaijevanger1992maximum} is obtained for the pure initial value problem
\begin{align*}
\frac{\partial u(t,x)}{\partial t}&= d\, \frac{\partial^2 u(t,x)}{\partial x^2}\,, \quad x\in\mathbb{R}, \, t\geq 0\,, 
\\[0.5ex]
  u(x,0)&=u_0(x)\, , \quad x\in\mathbb{R}\, . 
\end{align*}
In particular,  bound \eqref{contractive_inf} for contractivity is valid for all $m\geq 1$ 
\cite[Theorem 4.1]{kraaijevanger1992maximum} \cite{horvath1999maximum}. The bound $\tau/h^2\leq 2/3$ has also been obtained in 
\cite[Theorem 1]{farago2002sharpening} in the analysis  of the stability of CN method. In   \cite{farago2010non}  and \cite{horvath1999maximum}  results are based on the shape of the inverse  matrix $(I-\theta\tau B_h)^{-1}$ of dimension $m$, whose entries can be expressed in terms of hyperbolic functions \cite{Rozsa};  stepsize restrictions are obtained for each value of $m$, and bound  \eqref{contractive_inf} is valid for all $m$.

\subsection*{Contributions of the paper}
The approach followed in this paper consists on the representation of  the iteration matrix $A_m$ in \eqref{timeiteration} for the CN method 
 in terms of  three classes of polynomials, $P_n(x)$, $C_n(x)$ and $U_n(x)$, defined by iterations  \eqref{chebyshev22},  \eqref{chebyshev22G} and \eqref{chebyshev22R}, respectively; in particular,  $U_n(x)$  are the Chebyshev polynomials of the second kind. 
In the three cases the iteration process is the same, 
but different initial values are considered.  This formalism gives a new insight on the problem that allows us to improve some results in the literature.

The contributions of this paper are the following ones. 
With regard to positivity, for any number of grid points $m$, we have obtained that:
\begin{enumerate}
\item Crank-Nicolson method is positive if and only if $s=\tau/h^2\in  (0, s_m^{(p)}]$, with $s_m^{(p)}=1/(\cosh \omega_m^{(p)}-1)$, where $\omega_m^{(p)}$ is the unique positive root of equation \eqref{wroot_pos}. This root lies in the narrow interval $(\log(2+\sqrt{2}), \log(2+\sqrt{3})]\approx (1.22795, 1.31696]$.  Thus $\omega_m^{(p)}$ can be easily computed by solving \eqref{wroot_pos}   with  bisection method. 
Proposition \ref{proposition:p(x)} shows the connection between the bounds $s_m^{(p)}$
and polynomials $P_n(x)$.
Some values of $s_m^{(p)}$ are shown in Table \ref{Tab:pos}.  
\item The sequence of bounds $(s_m^{(p)})$  is  strictly monotonically increasing  with all the terms 
in the   interval $[1, 2(2-\sqrt{2}))$.
As a consequence,  the CN method does not
preserve  positivity when the  spatial mesh is refined (keeping $s$ constant).  
\item In the limit case, when $m$ tends to infinite, we recover  the known bound, $s\lesssim  1.17$ for positivity (\cite[p. 126]{hundsdorfer2003nst},\cite[Table 1]{farago2010non}). 
\end{enumerate}

With regard to contractivity, we have computed the value 
$\|A_m(s)\|_{\infty}$ for any number of grid points $m$ (see Figure \ref{Fig:normaAm}), and we have obtained that:
\begin{enumerate}
\item Crank-Nicolson method is contractive if and only if $s=\tau/h^2\in  (0, s_m^{(c)}]$, with 
\mbox{$s_m^{(c)}=\infty$}  for $m=1,2,3$. For $m\geq 4$,  $s_m^{(c)}=1/(\cosh \omega_m^{(c)} -1)$, where $\omega_m^{(c)}$ is the unique positive root of equation \eqref{ineqodd_0} or \eqref{ineqeven_0}, depending on the parity of $m$.
This root lies in the  interval  $\big[\log \big((3+\sqrt{5}+\sqrt{\smash[b]{-2+6 \sqrt{5}}})/4\big), \log 3\big)\approx \left(\left.0.767197, 1.09861\right.\right]$.  Thus $\omega_m^{(c)}$ can be easily computed by solving \eqref{ineqodd_0} or \eqref{ineqeven_0}  with the bisection method. Some values of $s_m^{(p)}$ are shown in Tables \ref{Tab:Con_odd} and \ref{Tab:Con_even} for odd and even values of $m$, respectively.
Some of these values can also be seen in Figure   \ref{Fig:normaAm}.
\item $\|A_m(s)\|_{\infty}<1$ for $s\in  (0, s_m^{(c)})$ and 
$\|A_m(s_m^{(c)})\|_{\infty}=1$  (see Figure \ref{Fig:normaAm}), 
that resembles property \eqref{cont_Ah} of the linear system \eqref{EDO_heat}. 
\item  For $m\geq 4$,  the sequence $(s_m^{(c)})$  is  strictly monotonically decreasing  with all the terms 
in the   interval $\big(3/2, 1+\sqrt{5}\,\big]$.
As a consequence,   CN method
preserves  contractivity when the  spatial mesh is refined (keeping $s$ constant). 
\item In the limit case, when $m$ tends to infinite, we recover  the known bound, $\tau/h^2\le 3/2$ for contractivity \cite[Th. 4.1(Q3); Section 7.1]{kraaijevanger1992maximum}, \cite[Eq. (14)]{horvath1999maximum}. 
\end{enumerate}

The results in this paper complete and improve some results in the literature  \cite{farago2010non,horvath1999maximum}. From equations \eqref{wroot_pos} and \eqref{ineqodd_0}-\eqref{ineqeven_0}, and  the associated intervals, the computation of  $s_m^{(p)}$ and $s_m^{(c)}$ for any value of $m$ is straightforward with the bisection method. Besides, we obtain equations to compute bounds $s_m^{(c)}$ both for odd and even values of $m$, whereas in~\cite{horvath1999maximum}, bounds $s_m^{(c)}$ are only given for even $m$. 
From our approach we also get the
correct value for $s_3^{(c)}$.  Figure \ref{Fig:sequences} illustrates the differences between positivity and contractivity of CN method for the heat problem \eqref{heat_eq}: if the scheme is positive for a given grid mesh $m$, then it is also contractive for any grid mesh.

\subsection*{Scope of the paper}
The rest of the paper is organized as follows. In Section \ref{sec:CN}, we explain the CN discretization process; notation and definitions are also given in this section. 
In Section \ref{sec:main} we show the main results of the paper,
namely: Theorems  \ref{Th:pos} and \ref{Th:con}; Table  \ref{Tab:pos}, containing  upper bounds $s_m^{(p)}$ for positivity; Tables \ref{Tab:Con_odd} and \ref{Tab:Con_even} containing  upper bounds   $s_m^{(c)}$  for contractivity (odd and even case); and Figure \ref{Fig:sequences} showing sequences $(s_m^{(p)})$ and $(s_m^{(c)})$.  
An illustrative example is also given in Section \ref{sec:main}.
Section \ref{sec:conclusions} contains some conclusions and ideas for future work. The proof of main results are given in Section \ref{sec_proofs}. Previously,  some technical material,  needed for the proofs in Section \ref{sec_proofs}, is included in Section \ref{sec:preliminar}.

\section{Crank-Nicolson method for the heat equation}\label{sec:CN}

In this paper we consider the Crank-Nicolson  method, a method of lines approach where second order central finite differences in space are followed by a second order time-stepping method.
Spatial discretization of heat equation  \eqref{heat_eq}  with second order  central finite differences  and mesh width $h=1/(m+1)$, gives the semi-discrete linear differential system 
\begin{equation} \label{EDO_heat}
w'(t)=B_h w(t)\, , \qquad w(0)=w_0\, , \,  t\geq 0\, , 
\end{equation}
where  $B_h=(d/h^2)\, \text{tridiag}(1,-2,1)$ is a matrix of dimension $m$,  $w(t)\approx \left(u(x_i, t)\right)_{i=1}^m$,  $w_0=\left(u_0(x_i)\right)_{i=1}^m$, and $x_i=i h$, 
$i=1, \ldots, m$, are the grid points.

As the diffusion  problem  \eqref{heat_eqEDP}-\eqref{heat_IC}  is positivity preserving \eqref{edp_pos} and monotonically decreasing  \eqref{edp_cont}, in order to obtain numerical approximations with these qualitative properties, problem \eqref {EDO_heat} should also be positivity preserving and contractive in the maximum norm. 

An initial value problem 
\begin{equation}\label{ODE_nl}
w'(t)=f(t, w(t)), \qquad w(t_0)=w_0\, ,  \, t\geq 0\, ,
\end{equation}
is called positivity preserving (positive for short) if $w_0\ge 0$ implies that  $w(t)\ge 0$ for $t\geq 0$, 
where the inequalities should be understood component-wise. Problem \eqref{ODE_nl}   is said to be contractive in the maximum norm if its solution $w(t)$ satisfy
$$\|w(t_2)\|_\infty\leq \|w(t_1)\|_{\infty} \quad \hbox{for} \, \, t_2 \geq t_1\geq 0\, . $$

It is well known that a linear problem, 
\begin{equation}\label{Ode_lin}
w'(t)=A w(t)\, , \qquad  w(t_0)=w_0\, , \, t\geq 0\, , 
\end{equation}
 where $A=(a_{ij})$ is an $m\times m$ matrix, is  positive if and only if $a_{ij}\geq 0$ for all $i\neq j$ \cite[Theorem 7.2]{hundsdorfer2003nst}. Matrix $B_h$ in \eqref{EDO_heat}  satisfies this condition and thus problem \eqref{EDO_heat} is positive. 
Observe that other spatial discretizations do not preserve positivity; indeed, there is 
an order barrier ($q\le 2$) from the requirement of positivity  \cite[p. 119]{hundsdorfer2003nst}.

Contractivity of solutions of the linear problem \eqref{Ode_lin} can be proven by using  the concept of logarithmic norm of matrix $A$. This concept  is an extremely useful tool to analyze the growth of solutions to ordinary differential equations because it  can take negative values. 
The solutions of problem \eqref{Ode_lin}   are of the form $w(t)=e^{At}w(0)$.
If we consider a vector norm and  its subordinate matrix norm, both  denoted   by $\|\cdot\|$,  then 
\begin{equation}
\|w(t)\|=\|e^{t A}w(0)\| \leq \|e^{t A}\| \, \|w(0)\|\, , \label{des_1}
\end{equation}
and contractivity is obtained if and only if  $\|e^{t A}\|\leq 1$.
Given the set
\[
{\cal M}=\left \{\delta \in \mathbb{R} \, | \, \| e^{tA}\|\leq  e^{t \delta} \, , t\geq 0\right\},
\]
it can be proven that $\mu_{\|\cdot\|}[A]=\min ({\cal M})$, where $\mu_{\|\cdot\|}[A]$ stands for the logarithmic norm of matrix $A$ in the norm ${\|\cdot\|}$ \cite[Proposition 2.1]{soderlind2006logarithmic}  (see, e.g., \cite{soderlind2006logarithmic,strom1975logarithmic,desoer1972measure} and the references therein for the definition and properties of logarithmic norms).  

From \eqref{des_1} and the definition of ${\cal M}$, we get the inequalities
\begin{equation*}
\|w(t)\|\leq e^{t\mu_{\|\cdot\|}[A]} \, \|w(0)\|\, , \quad t\geq 0\, , \qquad \quad \| e^{t A}\|\leq  e^{t\, \mu_{\|\cdot\|}[A]} \quad t\geq 0\, .
\end{equation*}
Thus,  if $\mu_{\|\cdot\|}A]\leq 0$, the zero solution is stable and $\| e^{t A}\|\leq  1$ for $t\geq 0$; 
if $\mu_{\|\cdot\|}[A] < 0$, then  the zero solution is exponentially stable and  $\|e^{t A}\|< 1$ for $ t> 0$ \cite[p. 634]{soderlind2006logarithmic}, \cite[p. 2]{kraaijevanger1992maximum}.

For the maximum norm, the logarithmic norm of a matrix $A=(a_{ij})$  is given by
$$\mu_{\infty}[A]= \max_{1\leq i \leq n} \left(a_{ii} + \sum_{j=1}^n |a_{ij}|\right)\, . 
$$
In particular, for matrix $B_h$ in \eqref{EDO_heat}, as
$$ a_{ii} + \sum_{j=1}^n |a_{ij}| = \begin{cases}
-1\, , & i=1, m\, ,\\
0\, , & i=2, \ldots, m-1\, ,
\end{cases}
 $$ 
we get $\mu_\infty[B_h]=0$, and thus
\begin{equation}\label{cont_Ah}
\|e^{t B_h}\|_{\infty}\leq 1, \qquad t\geq 0\, , 
\end{equation}
that ensures that problem  \eqref{EDO_heat} is contractive in the maximum norm, that is,
$$\|w(t)\|_{\infty}\leq  \|w(0)\|_{\infty}\, , \quad t\geq 0\, .$$

The time stepping process in Crank-Nicolson method with constant time step $\tau$, gives the iteration
\begin{equation*}
w_{n}=\phi(\tau B_h)w_{n-1}\, , \quad n\ge1\,,
\end{equation*} 
where 
\begin{equation}\label{sf_trapezoidal}
\phi(z)=\frac{1+\frac{1}{2}z}{1-\frac{1}{2}z}
\end{equation}
 is the stability function of the time integrator. On the following, we denote 
\mbox{$A_m=\phi(\tau B_h)$} 
 to the Crank-Nicolson iteration matrix of dimension $m$, that is,
\begin{equation}\label{aeAt} 
 A_{m}=\left(I_m-\tfrac{\tau}{2} B_h\right)^{^{\!\!\!-1}}\!\!\!\left(I_m+\tfrac{\tau}{2} B_h\right)=
\setlength{\arraycolsep}{2pt}
\begin{pmatrix}
1+s&-\frac{s}{2} &&    \\[1ex]
 -\frac{s}{2}&1+ s&\ddots &\\[1ex]
&\ddots&\ddots &-\frac{s}{2}\\[1ex]
 &&-\frac{s}{2}&1+s
\end{pmatrix}^{^{\!\!\!\!\!\!-1}}\!\!\!
\begin{pmatrix}
1-s&\frac{s}{2} &&    \\[1ex]
 \frac{s}{2}&1- s&\ddots &\\[1ex]
&\ddots&\ddots &\frac{s}{2}\\[1ex]
 &&\frac{s}{2}&1-s
\end{pmatrix}.
\end{equation}%
Observe that the two matrices in \eqref{aeAt}, corresponding to half step with  forward Euler and half step with backward Euler, commute  because of the the linearity of the system \eqref{EDO_heat}. 
Besides, positivity and contractivity can be studied by analyzing these properties for forward and backward Euler separately. 

Although there are no restrictions for positivity and contractivity  with backward Euler applied to system \eqref{EDO_heat}, the   restriction for positivity and contractivity   with forward Euler is  
$s\le 1$\,. This stepsize restriction for positivity is not sharp for problem  \eqref{EDO_heat}; numerical experiments in  \cite[p.126]{hundsdorfer2003nst} show that numerical positivity can be obtained for $s\lesssim 1.17$. 
As it has been pointed out above, a closer look at the iteration matrix $A_m$ in \eqref{thetamethod} 
or \eqref{aeAt}, gives sharper  bounds.

\section{Main results}\label{sec:main}
In this section we show the main results of the paper concerning stepsize restrictions for positivity and contractivity in the maximum norm for the $m$-dimensional system \eqref{EDO_heat}. The proofs require some preliminary material about the  structure of matrix $A_{m}$  and are   given in section \ref{sec_proofs}.

On the following theorems,   the positivity of the matrix  $A_{m}$ means that all the entries of the matrix are non-negative; similarly, the contractivity in the maximum norm of the matrix  $A_{m}$ means $\norm{A_m}_\infty\le 1$.

\begin{theorem} (Positivity of Crank Nicolson method) \label{Th:pos}\mbox{}
\begin{enumerate}
\item For $m\in \mathbb{N}$, the matrix $A_{m}(s)$ in \eqref{aeAt} is positive if and only if 
 \begin{equation} \label{sCFL_pos}
 s\le  s_m^{(p)}:=\frac{1}{\cosh \omega_m^{(p)} -1} \,,
 \end{equation} 
where $\omega_m^{(p)} {\in \big( \log(2+\sqrt{2}),  \log(2+\sqrt{3}) \big]}$ is the unique positive root of equation  
\begin{equation}\label{wroot_pos}
\coth(m\omega)\, {\sinh\omega}={3 \cosh\omega -4}\, . 
\end{equation}
\item The sequence $(s_m^{(p)})$  is   strictly monotonically increasing  with all the terms 
in the narrow interval $ \big[1, 2(2-\sqrt{2})\big).$ As a consequence, Crank Nicolson method preserves positivity when the spatial mesh is refined (keeping $s$ constant).
\end{enumerate}
\end{theorem}

\begin{remark}\mbox{}
\begin{enumerate}
\item The sequence $(s_m^{(p)})$  increasingly converges to the limit value $s_\infty^{(p)}:=2(2-\sqrt{2})$ (see Table \ref{Tab:pos}  and Figure \ref{Fig:sequences}).  This  value was also obtained in 
\cite{farago2010non} with other techniques. 
Consequently, if  
\begin{equation}\label{cota_positividad}
s< s_\infty^{(p)}=2(2-\sqrt{2})\approx 1.17\,,
\end{equation}
then there exists a natural number $m_0$ such that the matrix $A_{m}$   is positive for any value of $m\ge m_0$.
\item As $\omega_m^{(p)} {\in \big( \log(2+\sqrt{2}),  \log(2+\sqrt{3}) \big]}\approx \left(\left.1.22795, 1.31696\right.\right]$, an approximated  value can be easily computed by the bisection method. 
\end{enumerate}
In Table \ref{Tab:pos} below we show the roots  
$\omega_m^{(p)}$ of equation \eqref{wroot_pos} and the CFL restrictions $s_m^{(p)}$ for positivity  in  \eqref{sCFL_pos}  for different values of $m$.
\end{remark}   
\begin{table}[ht]  
\centering
\begin{tabular}{|c|c|c|c|} 
      \hline
      $m$ &\phantom{\rule[0cm]{0cm}{0.5cm}} 
      $\omega_m^{(p)}$ & $x_m^{(p)}=\cosh \omega_m^{(p)}$ & $s_m^{(p)}=1/(x_m^{(p)}-1)$ \\
      \cline{1-4}
      1 &\phantom{\rule[-0.1cm]{0.0cm}{0.53cm}}%
       $\log(2 + \sqrt{3})$ & $2$& 1\\
      \cline{1-4}
      2 &\phantom{\rule[-0.1cm]{0.0cm}{0.53cm}}%
      1.23590& $1+\sqrt{3}/2\approx 1.86603$&
      $2/\sqrt{3}\approx 1.15470$  \\
      \cline{1-4}
     3 & 1.22864&1.85464&1.17009\\      
      \cline{1-4}
     4 & 1.22801&1.85365&1.17144\\      
      \cline{1-4}      
      \vdots &  \vdots&\vdots  &\vdots \\
      \cline{1-4}
      7 & 1.22795& 1.85355& 1.17157\\
      \cline{1-4}
      \vdots &  \vdots&\vdots  &\vdots \\
      \cline{1-4}
      $\infty$ &\phantom{\rule[-0.1cm]{0.0cm}{0.53cm}}%
       $\log(2 + \sqrt{2})$ & 
      $(6+\sqrt{2})/4$  &
      $2(2-\sqrt{2})\approx 1.171572875$ \\ 
      \cline{1-4}
    \end{tabular}
          \caption{Roots $\omega_m^{(p)}$ of \eqref{wroot_pos} and CFL restrictions $s_m^{(p)}$ in \eqref{sCFL_pos} for positivity.}\label{Tab:pos}
  \end{table}

Next, we give the results for contractivity  in the infinite norm. Observe that the symmetry of  matrix ${A_m}$ makes 
$\norm{A_m}_\infty=\norm{A_m}_1$, and the result is also valid for the 1-norm.

\begin{theorem}(Contractivity of Crank Nicolson method)   \label{Th:con} \mbox{}
\begin{enumerate}
\item For $m\in\{1,2,3\}$ the matrix $A_{m}(s)$ in \eqref{aeAt} is contractive in the maximum norm  for any value of $s>0$.
\item For $m\in \mathbb{N}$, $m\ge 4$\,, the matrix $A_{m}(s)$ in \eqref{aeAt} is contractive in the maximum norm if and only if 
 \begin{equation}\label{iff4}
 s\le  s_m^{(c)}:=\frac{1}{\cosh \omega_m^{(c)} -1} \,,
 \end{equation} 
where $\omega_m^{(c)}\in \big[\log \big((3+\sqrt{5}+\sqrt{\smash[b]{-2+6 \sqrt{5}}})/4\big), \log 3\big)%
\approx[0.767197,1.09861)$ is the uni\-que positive root of equation  
\begin{equation}\label{ineqodd_0}
2\sinh\frac{(m-1)\omega}{4}\, \sinh\frac{(m+1)\omega}{4}= \sinh\frac{\omega}{2}\, \sinh\frac{(m+1)\omega}{2}\,,
\end{equation}
 if $m$ is odd, or equation
{\footnotesize \begin{equation}\label{ineqeven_0}
\sinh^2\left(\frac{\omega}{2}\right)   \sinh\frac{m\omega}{2}\left(\sinh\frac{(m+2)\omega}{2}-\sinh\frac{m\omega}{2}\right)
= \sinh \omega \, \sinh\frac{(m+1)\omega}{2}\sinh\frac{m\omega}{4}\sinh\frac{(m-2)\omega}{4} \, , 
\end{equation}}
if $m$ is even. 
\item The sequence $(s_m^{(c)})$  is  strictly monotonically decreasing    with all the terms 
in the   interval $\big(3/2, 1+\sqrt{5}\,\big]$.
As a consequence,   Crank Nicolson method   does not 
preserve contractivity when the  spatial mesh is refined (keeping $s$ constant). 
\end{enumerate}
\end{theorem}

\begin{remark} 
The sequence $(s_m^{(c)})$  decreasingly converges to the limit value $s_\infty^{(c)}:=3/2$ (see   Figures  \ref{Fig:sequences} and \ref{Fig:normaAm}). Consequently, the matrix $A_m(s)$ is contractive for all $m$ if and only if 
\mbox{$s\in (0,3/2]$}. 
For infinite matrices the bound $s_\infty^{(c)}:=3/2$ has been obtained with different techniques in \cite{farago2002sharpening,kraaijevanger1992maximum}. 
\end{remark}

In Tables  \ref{Tab:Con_odd} and \ref{Tab:Con_even} we give the CFL restrictions $s_m^{(c)}$ in \eqref{iff4} for contractivity in the infinite norm. The  values in Table \ref{Tab:Con_odd} (odd case) and  Table \ref{Tab:Con_even} (even case)  have been obtained from equations \eqref{ineqodd_0} and \eqref{ineqeven_0}, respectively. Observe that both, {the roots $\omega_m^{(c)}$ of equation \eqref{ineqodd_0} (odd case) and the roots $\omega_m^{(c)}$ of equation \eqref{ineqeven_0} (even case), increasingly converge to the limit value
$\log 3$.  Consequently these roots are in the narrow 
interval~${\big[\log \big((3+\sqrt{5}+\sqrt{\smash[b]{-2+6 \sqrt{5}}})/4\big), \log 3\big)\approx [0.767197,1.09861)}$
and can be easily obtained with bisection method. 
The numeric values shown in tables \ref{Tab:Con_odd} and \ref{Tab:Con_even} were obtained after 10 iterations with bisection method.
\begin{figure}[ht]
\begin{picture}(0,170)(0,0)%
\put(100,0){\includegraphics[scale=0.34]{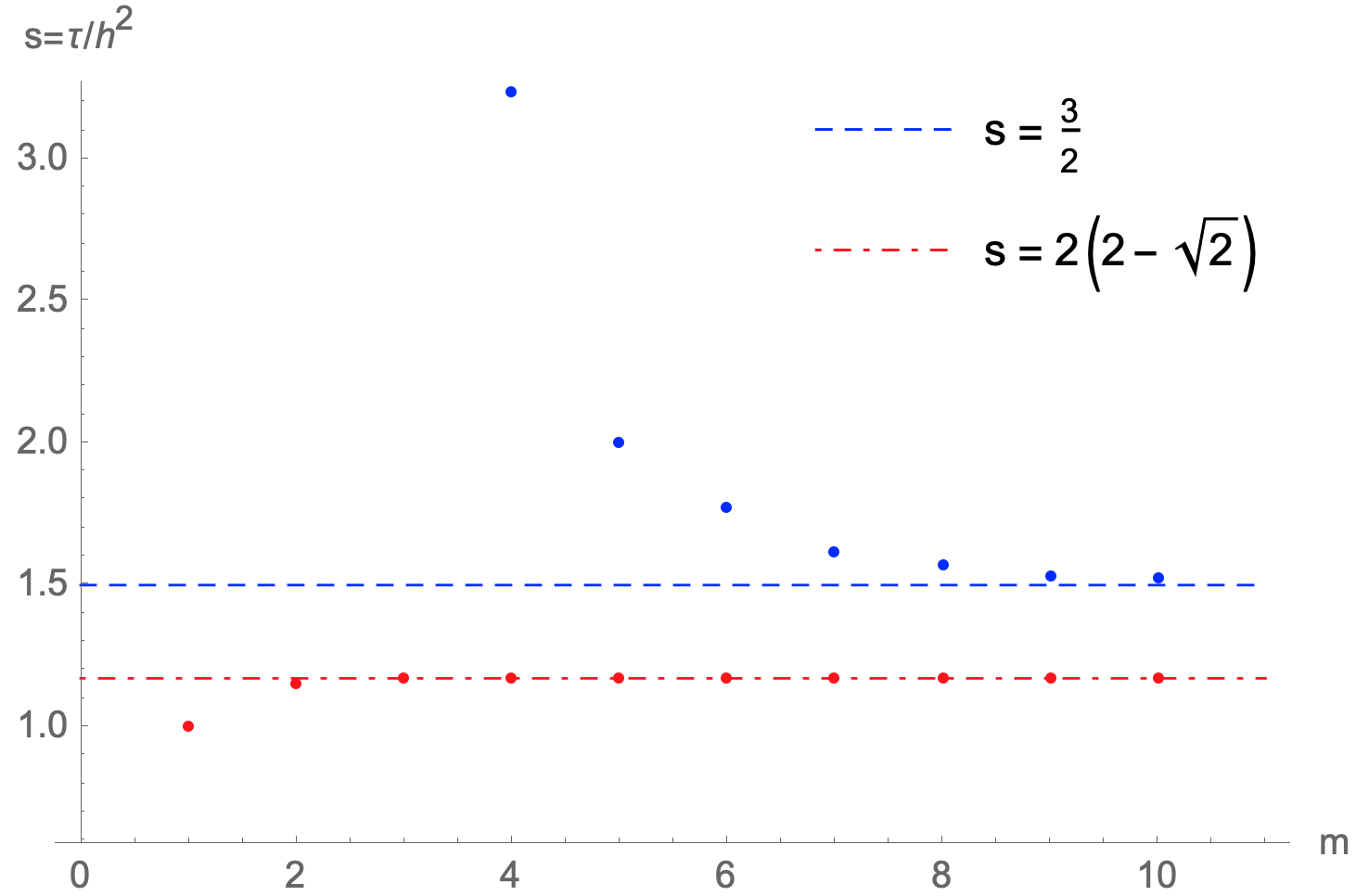}}
\put(200,100){\textcolor{blue}{$s_m^{(c)}$}}
\put(145,20){\textcolor{red}{$s_m^{(p)}$}}
\end{picture}
\caption{Sequences 
$s_m^{(p)}$  and $s_m^{(c)}$ 
with the restrictions over CFL coefficient $s$   for positivity and contractivity.}\label{Fig:sequences}
\end{figure}

\begin{corollary}
Consider the numerical integration of the $m$-dimensional problem \eqref{EDO_heat} with the Crank-Nicolson method. Then
if the  method is positive for a given CFL coefficient, then it is also contractive in the infinity  and 1 norms. 
\end{corollary}
\begin{proof}
It is straightforward from Theorems \ref{Th:pos} and \ref{Th:con}. 
\end{proof}
The converse to the previous Corollary is not true as we can see in the following trivial example, where contractivity is preserved while positivity is violated.

\begin{example}
Consider the diffusion equation in \eqref{heat_eq} with initial function
\[
u(x,0)=\left\{\begin{array}{ll}
0 & \text { for } 0<x<\frac{7}{8}\,, \\ 
1 & \text { for } \frac{7}{8} \leq x<1\,,
\end{array}\right.
\]
giving discontinuities at $x=7/8$ and $x=1$ for t=0. 
From second order central differences with $h=1/8$ we get approximations $\omega(t)=(\omega^1(t),
\ldots,\omega^7(t))\approx (u(x_1,t),\ldots,u(x_7,t))$. 
Application with $\tau=0.025$ of one Crank-Nicolson step,
$w_{1}=A_7\, w_{0}$, 
 gives  the vector $w_{1}\approx \omega(\tau)$
\begin{equation}\label{NOpositivity}
w_{1}=(0.0013,0.0041,0.0120,0.0356,
              0.1019,0.2961,-0.1397)\,,
\end{equation} 
where $A_7$ is the Crank-Nicolson iteration matrix in 
\eqref{aeAt} for the case $m=7$ and $w_{0}$ is the initial profile $w_{0}=(0,0,0,0,0,0,1)$. Observe that 
$\norm{w_{1}}_{\infty}=0.2961<\norm{w_{0}}_{\infty}=1$.

In this example the CFL coefficient 
$s=\tau/h^2=1.6$ is greater than the 
positivity bound $s_7^{(p)}=1.17157$ (see Table \ref{Tab:pos}), but it is lower 
than the contrativity  one \mbox{$s_7^{(c)}=1.61803$}~(see Table \ref{Tab:Con_odd}). Consequently, contractivity is preserved while we cannot ensure positivity. Actually, as we can see in vector 
$w_{1}$ in \eqref{NOpositivity}, negativity is not preserved.
\end{example}
\begin{table}[ht]
\centering 
\begin{tabular}{|c|c|c|c|} 
      \hline
      $m$ &\phantom{\rule[0cm]{0cm}{0.5cm}}
      $\omega_m^{(c)}$ & $x_m^{(c)}=
      \cosh \omega_m^{(c)}$
      & $s_m^{(c)}=1/(x_m^{(c)}-1)$ \\  \hline
      3 &   &  &  $\infty$ \\       \hline
      5 & $2\operatorname{arccsch}2\approx 0.962424$
      & 3/2 & 2\\     \hline
      7 &\phantom{\rule[-0.25cm]{0cm}{0.9cm} }$
      \log \frac{\bigl(1+\sqrt{5}+\sqrt{2(1+\sqrt{5})}\,\bigr)}{2}
      \approx 1.06131$& $(1+\sqrt{5})/2$& 
      $(1+\sqrt{5})/2\approx 1.61803$\\  \hline
       9 & $1.08707$& $1.65139$& $1.53518$\\ \hline
      \vdots &  \vdots&\vdots  &\vdots \\   \hline
      $\infty$ & $\log 3\approx 1.09861$ & $5/3$ &
      $3/2$  \\     \hline
    \end{tabular}
    \caption{Positive root of \eqref{ineqodd_0} and 
    bounds for contractivity  (odd case).}
    \label{Tab:Con_odd}
  \end{table}
 \begin{table}[h!]
\centering
\begin{tabular}{|c|c|c|c|} 
      \hline
      $m$ &\phantom{\rule[0cm]{0cm}{0.5cm}}
       $\omega_m^{(c)}$ & $x_m^{(c)}=
       \cosh \omega_m^{(c)}$ 
       & $s_m^{(c)}=1/(x_m^{(c)}-1)$ \\  \hline
      4 &\phantom{\rule[-0.25cm]{0cm}{0.8cm} } 
      $ \log \bigl(\frac{1}{4}(3+\sqrt{5}+\sqrt{\smash[b]{-2+6 \sqrt{5}}})\bigr)$ & $\frac{1}{4}(3+\sqrt{5})$ &  $1 + \sqrt{5}$ \\
        & $\approx 0.767197$ & $\approx 1.30902$ &  $\approx 3.23607$ \\
        \hline
            \vdots &  \vdots&\vdots  &\vdots \\
            \hline
      10 & $1.09110$& $1.65669$& $1.52278$\\      
      \hline  
            \vdots &  \vdots&\vdots  &\vdots \\
            \hline
      20 & $1.09855$& $1.66658$& $1.5002$\\
      \hline
            \vdots &  \vdots&\vdots  &\vdots \\
            \hline
      $\infty$ & $\log 3\approx 1.09861$ & $5/3$  &
      $3/2$  \\
      \hline
    \end{tabular}
     \caption{Positive root of \eqref{ineqeven_0} and bounds for contractivity  (even case).}\label{Tab:Con_even}
  \end{table}

\section{Conclusions and future work}\label{sec:conclusions}
In this paper we have studied CFL restrictions when the Crank-Nicolson method is used to solve the heat equation \eqref{heat_eq} with Dirichlet boundary conditions.
We have obtained bounds $s_m^{(p)}$ for positivity and 
bounds $s_m^{(c)}$ for contractivity for any value of the spatial discretization parameter $m$. To get these bounds we have represented the Crank-Nicolson iteration matrix $A_m$ in terms of some Chebyshev-like polynomials 
(\ref{chebyshev22},\ref{chebyshev22G},\ref{chebyshev22R}). We have obtained  bounds   
for the $\theta$-method \eqref{thetamethod} for the particular case $\theta=1/2$, but similar bounds can be obtained for other values of the parameter following the same ideas.

We have seen that the positivity of matrix   $A_m$ is determined by the largest root of polynomial $P_m(x)$, and we have provided a narrow interval where this root can be found. Similarly, we have considered these polynomials to analyze the contractivity and we have provided a narrow interval to get the corresponding bounds, both in the odd and even case. 

As far as we know, polynomials $P_n(x)$ and $C_n(x)$ have not been used previously in the literature. The strength of this idea can be used to prove qualitative properties for other problems.  
       Furthermore,  this approach can also be used for other discretizations of the heat equation \eqref{heat_eqEDP} \cite{HiRoPeriodicBC}.

\section{Preliminary material for the proofs of the main results}
\label{sec:preliminar} 
In this section we introduce the notation, definitions and some results needed to prove the main results of the paper. In subsection \ref{subsect:Am} we express the  
Crank-Nicolson iteration matrix  \eqref{aeAt} in terms of   rational functions. These 
functions can be  written easily with the help of some 
 Chebyshev-like polynomials $U_m$, $P_m$ and $C_m$. In  subsection \ref{sec:polynomials} we give the  definition    of these polynomials and we also add some results   that  will be used in the proofs of Section~\ref{sec_proofs}.

\subsection{The Crank-Nicolson matrix $A_m$ in terms of rational functions} \label{subsect:Am}
A direct computation of the product 
$(I_m-\frac{\tau}{2}B_h)^{-1}(I_m+\frac{\tau}{2} B_h)$ 
in \eqref{aeAt} gives us the entries of matrix $A_m$ 
expressed as rational functions, where the polynomials involved can be obtained recursively. These   simplified closed  expressions  will make it easier to get bounds for positivity and contractivity. 
\begin{example} \label{Ex:m3}
For $m=3$, a direct computation of the symmetric matrix $A_{3}$ in \eqref{aeAt} gives
\begin{align}
A_{3}(s)&=\left(
\begin{matrix}
\frac{2+2 s-2 s^{2}-s^{3}}{2+6 s+5 s^{2}+s^{3}} & \frac{2 s}{2+4 s+s^{2}} & \frac{s^{2}}{2+6 s+5 s^{2}+s^{3}} \\[1ex]
\frac{2 s}{2+4 s+s^{2}} & \frac{2-s^{2}}{2+4 s+s^{2}} & \frac{2 s}{2+4 s+s^{2}}\nonumber \\[1ex]
\frac{s^{2}}{2+6 s+5 s^{2}+s^{3}} & \frac{2 s}{2+4 s+s^{2}} & \frac{2+2 s-2 s^{2}-s^{3}}{2+6 s+5 s^{2}+s^{3}}
\end{matrix}\right)\,. 
\end{align}
Remember $s={\tau}/{h^2}$ denotes CFL coefficient.
This matrix can  be written even simpler if we consider the new variable $x=1+1/s$. Observe that $x>1$ when $s>0$. With the help of a new kind of polynomials   $U_n(x)$, $P_n(x)$ and $C_n(x)$, we can write $A_3(x)$ as
\[
 A_{3}(x)=\left(\begin{matrix}
\frac{2 x^{3}-4 x^{2}+1}{2 x^{3}-x} & \frac{2(x-1)}{2 x^{2}-1} & \frac{x-1}{2 x^{3}-x} \\[1ex]
\frac{2(x-1)}{2 x^{2}-1} & \frac{2 x^{2}-4 x+1}{2 x^{2}-1} & \frac{2(x-1)}{2 x^{2}-1} \\[1ex]
\frac{x-1}{2 x^{3}-x} & \frac{2(x-1)}{2 x^{2}-1} & 
\frac{2 x^{3}-4 x^{2}+1}{2 x^{3}-x}
\end{matrix}\right)=
 \frac{1}{U_3(x)}\left(\begin{matrix}
{P_3}(x) & {C_2(x)} & {C_1(x)} \\[1ex]
{C_2(x)} & {C_1(x)+P_3(x)} & {C_2(x)} \\[1ex]
{C_1(x)} & {C_2(x)} & {P_3(x)}
\end{matrix}\right). 
\]
Observe that $A_3(x)$ has been written just in terms of 
$U_3(x)$, $P_3(x)$,  $C_1(x)$ and $C_2(x)$. 
We give the definition and all the details about  these  polynomials $U_n(x)$, $P_n(x)$ and $C_n(x)$ in the next subsection.  Before, we extend the ideas in this simple example to the more general case of the matrix $A_m(x)$ for any value of $m$, although we have to distinguish between the odd case and the even case.
\end{example} 
\begin{proposition}
Matrix $A_m(x)$ can be written in terms of polynomials 
$U_m(x)$,  $P_m(x)$ and $C_n(x)$,  $n=1, \ldots, m-1$. 
If $m$ is an odd number, Crank Nicolson matrix 
can be reduced to
\begin{equation}\label{Aimpar}
A_{m}(x)=\frac{1}{U_m}\left(\begin{smallmatrix}
{P_m} & {C_{m-1}}& 
\ldots & {C_{\frac{m+1}{2}}} & \ldots &{C_2} & {C_1} \\[1ex]
{C_{m-1}} & {P_m+C_{m-2}} & 
\ldots & {C_{\frac{m-1}{2}}+C_{\frac{m+3}{2}}}&\ldots&{{C_1+C_3}}& {C_2} \\ 
\vdots & \vdots & \ddots&    \vdots  
   & \reflectbox{$\ddots$}& \vdots& \vdots\\[1ex]
{C_{\frac{m+1}{2}}} & {C_{\frac{m-1}{2}}+C_{\frac{m+3}{2}}} 
& \ldots & P_m+\sum\limits_{n=1}^{\frac{m-1}{2}}C_{2n-1}
 & \ldots &   {C_{\frac{m-1}{2}}+C_{\frac{m+3}{2}}} &{C_{\frac{m+1}{2}}} \\[1ex]
\vdots & \vdots & \reflectbox{$\ddots$}& \vdots      
& \ddots & \vdots& \vdots \\[1ex]
{C_{2}} & {C_1+C_3} 
& \ldots & {C_{\frac{m-1}{2}}+C_{\frac{m+3}{2}}} & \ldots & {P_m+C_{m-2}}& C_{m-1}\\[1ex]
{C_{1}} & {C_{2}} & 
\ldots & {C_{\frac{m+1}{2}}}  &\ldots &{C_{m-1}} & {P_m}
\end{smallmatrix}\right)\,,
\end{equation}
where all the polynomials are evaluated at $x=1+1/s$. 
If $m$ is an even number, we write
\begin{equation}\label{Apar}
A_{m}(x)=\frac{1}{U_m}\left(\begin{smallmatrix}
{P_m} & {C_{m-1}}& 
\ldots &{C_{\frac{m+2}{2}}}& {C_{\frac{m}{2}}} &  \ldots&{C_2} & {C_1} \\[1ex]
{C_{m-1}} & {P_m+C_{m-2}} &
\ldots & {C_{\frac{m}{2}}+C_{\frac{m+4}{2}}}&{C_{\frac{m-2}{2}}+C_{\frac{m+2}{2}}}&\ldots&{{C_1+C_3}}& {C_2} \\ 
\vdots & \vdots & \ddots
& \vdots&\vdots
   & \reflectbox{$\ddots$}& \vdots& \vdots\\[1ex]
{C_{\frac{m+2}{2}}} & {C_{\frac{m}{2}}+C_{\frac{m+4}{2}}} & 
%
 \ldots & P_m+\sum\limits_{n=1}^{\frac{m-2}{2}}C_{2n} & 
\sum\limits_{n=1}^{\frac{m}{2}}C_{2n-1} &\ldots &   {C_{\frac{m-2}{2}}+C_{\frac{m+2}{2}}} &{C_{\frac{m}{2}}} \\[1ex]
{C_{\frac{m}{2}}} & {C_{\frac{m-2}{2}}+C_{\frac{m+2}{2}}} & 
 \ldots &\sum\limits_{n=1}^{\frac{m}{2}}C_{2n-1}     & 
P_m+\sum\limits_{n=1}^{\frac{m-2}{2}}C_{2n}
& \ldots &  {C_{\frac{m}{2}}+C_{\frac{m+4}{2}}} &{C_{\frac{m+2}{2}}} \\[1ex]
\vdots & \vdots &
    & \vdots &  
& \ddots & \vdots& \vdots \\[1ex]
{C_{2}} & {C_1+C_3} & 
\ldots & {C_{\frac{m-2}{2}}+C_{\frac{m+2}{2}}} &{C_{\frac{m}{2}}+C_{\frac{m+4}{2}}} & \ldots &  {P_m+C_{m-2}}& C_{m-1}\\[1ex]
{C_{1}} & {C_{2}} & 
\ldots & {C_{\frac{m}{2}}}& {C_{\frac{m+2}{2}}} & \ldots &{C_{m-1}} & {P_m}
\end{smallmatrix}\right)\,.
\end{equation}
\end{proposition}

\begin{proof}
It is straightforward from the   computation of the product 
$(I_m-\frac{\tau}{2}B_h)^{-1}(I_m+\frac{\tau}{2} B_h)$ 
in \eqref{aeAt} and the use of polynomials $U_n(x)$,  $P_n(x)$ and $C_n(x)$,  $n=1, \ldots, m-1$, defined in the next subsection.
\end{proof}
Observe that   $A_{m}(x)$ is  bisymmetric, that is, it is symmetric on both diagonals. This implies that  $A_{m}(x)$ is also centrosymmetric. Then the entries  $a_{ij}$ satisfy
$a_{ij}= a_{n-i+1,n-j+1}$\,, for $1\leq i,j \leq n$. 
Consequently, if $m$ is odd, the number of different entries in matrix  $A_m$ is $1+3+5+\cdots+m=(m+1)^2/4$, and, if $m$ is even, this number is $2+4+6+\cdots+m=(m/2+1)m/2$. For example, the number of different elements in matrix $A_3$  in Example \ref{Ex:m3} is 4, while this number is 6 for matrix $A_4$ in Example \ref{Ex:m4} below.

Observe also that the numerator of each entry 
$a_{ij}$ in   matrix $A_{m}(x)$ is a sum of some polynomials
$P_n(x)$, $C_n(x)$,  $n=1, \ldots, m-1$, 
 and the number of polynomials in this sum is equal to $\min\{i,j,m-i+1,m-j+1\}$.  
Properties of polynomials $U_n(x)$, $P_n(x)$ and $C_n(x)$ will allow us to analyze positivity and contractivity of Crank Nicolson method in a quite simple way. In the next section, we study these properties. 
\begin{example}\label{Ex:m4}
In Example \ref{Ex:m3} we have considered the odd case $m=3$. Here, for completeness, we consider the even case $m=4$. A direct computation of the symmetric matrix $A_{4}$ in~\eqref{aeAt} gives
\begin{align*}
A_{4}(s)&=\frac{1}{u_4(s)}
\left(\begin{smallmatrix}
p_4(s) & 4 s\left(4+8 s+3 s^{2}\right) & 8 s^{2}(1+s) & 4 s^{3} \\
4 s\left(4+8 s+3 s^{2}\right) & \hspace{0.15cm}16+32 s+4 s^{2}-16 s^{3}-5 s^{4}  & 16 s(1+s)^{2} & 8 s^{2}(1+s) \\
8 s^{2}(1+s) & 16 s(1+s)^{2} &\hspace{-0.35cm}16+32 s+4 s^{2}-16 s^{3}-5 s^{4}  &\hspace{0.15cm} 4 s\left(4+8 s+3 s^{2}\right) \\
4 s^{3} & 8 s^{2}(1+s) & 4 s\left(4+8 s+3 s^{2}\right) & \hspace{0cm}p_4(s)
\end{smallmatrix}%
\right),
\end{align*}
where $p_4(s)=-5 s^4-24 s^3-4 s^2+32 s+16$
and  $u_4(s)=5 s^4+40 s^3+84 s^2+64 s+16$\,. 
Now, with the help of variable $x=1+1/s$, we can write

\begin{align*}
A_{4}(x)&=
\frac{1}{U_4(x)}\left(\begin{matrix}
{P_4(x)} & {C_3(x)}& {C_2(x)} & {C_1(x)} \\[1ex]
{C_3(x)} & {P_4(x)}+{C_2(x)}& {C_1(x)}+{C_3(x)} & {C_2(x)} \\[1ex]
{C_2(x)} & {C_1(x)}+{C_3(x)} &{P_4(x)}+{C_2(x)}& {C_3(x)} \\[1ex]
{C_1(x)} & {C_2(x)} & {C_3(x)}&{P_4(x)}
\end{matrix}\right). 
\end{align*}
Observe that there are two central rows  in the even case,  but just one in the odd case.
\end{example}

\subsection{Polynomials $U_n$, $P_n$ and $C_n$}\label{sec:polynomials}
In this section we define the new polynomials $P_n(x)$ and $C_n(x)$. Together with the help of
Chebyshev polynomials of   the second kind $U_n(x)$ \cite{gross2016root,liu2007unified,qi2019some}, we have got a simple way of writing Crank-Nicolson   matrix~$A_m(x)$. 
Besides, here we give some results concerning these polynomials, with particular interest in the distribution of their roots. Chebyshev polynomials of   the second kind $U_n(x)$ belong to a  general class of orthogonal polynomials and there are many works about the behaviour of their zeros  \cite{gross2016root,liu2007unified}. However, polynomials 
$P_n(x)$ and $C_n(x)$ do not belong to this class of orthogonal polynomials and, as far as we know, nothing is known about their roots.

\subsubsection*{Chebyshev polynomials of   second kind}
Chebyshev polynomial of   second kind of degree $n\ge 0$ is defined as
\[
U_n(x) = \frac{\sin ((n+1) \arccos x)}{\sin (\arccos x)}, \quad x\in [-1, 1]\,,
\]
or, in the angle variable $\omega$,
$U_n(\cos\omega) = {\sin ((n+1)\omega)}/{\sin\omega}$, \ 
$\omega\in[0,\pi]$\,.
These polynomials  can also be defined for any value of $x\in\mathbb{R}$ by the recurrence relation
\begin{align} 
U_0(x)&=1\,,\nonumber\\
U_1(x)&=2x\,,\nonumber\\
U_n(x)&=2x \, U_{n-1}(x)- U_{n-2}(x)\,.\label{chebyshev22}
\end{align}  
It is possible to write the  recurrence relation \eqref{chebyshev22}   in terms of the determinant of the tridiagonal matrix   $\text{tridiag}(1,2x,1)$   of dimension $n$.
\begin{equation} \label{chebyshev2det2}
U_n(x)= \left|\begin{array}{cccc}
2x &   1 &  &  \\
  1& 2x & \ddots& \\
   & \ddots &\ddots &1 \\
   && 1 & 2x
\end{array}\right|
\end{equation} 
Recall that each polynomial $U_n(x)$ has   $n$ roots 
$x_{i}^n=\cos \left({i\pi}/{(n+1)} \right)$\,,
\mbox{$i=1, \ldots, n$}, in the interval $[-1,1]$. These roots are uniformly distributed in the angle 
variable $\omega=\arccos x$ in the interval $[0,\pi]$. 
Notice that polynomials $U_n(x)$ defined in~\eqref{chebyshev22} are positive for $x>1$.

\subsubsection*{Polynomials $P_n$} 
If we change the first two elements  in the recursive  relation  
 \eqref{chebyshev22}, then a new family of polynomials  can be defined
\begin{align} 
P_0(x)&=-1\,,\nonumber\\
P_1(x)&=2x-4\,,\nonumber\\
P_n(x)&=2x\, P_{n-1}(x)- P_{n-2}(x) \,,\label{chebyshev22G}
\end{align} 
where $P_n(x)$ denotes the polynomial of degree $n$. In this case the recurrence relation \eqref{chebyshev22G} can also be written in terms of the determinant of a   matrix   
of dimension $n$.
\begin{equation} \label{chebyshev2det3}
P_n(x)= \left|\begin{array}{ccccc}
2x-4 &   -1 &  & & \\
  1& 2x & 1&& \\
   & \ddots &\ddots &\ddots  & \\
   &   &   1 & 2x&1  \\  
   &&& 1 & 2x
\end{array}\right|\,.
\end{equation}

In the next proposition we analyze  the roots of each   polynomial $P_n(x)$.  In the proof, shown in Section \ref{sec_proofs}, it is relevant the fact that 
each polynomial $P_n(x)$ can be written in terms of 
Chebyshev polynomials of   second kind 
\begin{equation} \label{PtoU}
P_n(x)=2U_{n-2}(x)-4U_{n-1}(x)+U_{n}(x)\,.
\end{equation} 
This equality is obtained by writing the 
determinant~\eqref{chebyshev2det3} in terms of the determinant~\eqref{chebyshev2det2}, and the use of the recurrence relation \eqref{chebyshev22}.

\begin{proposition}\label{proposition:p(x)}\mbox{}
The polynomial $P_n(x)$ defined in \eqref{chebyshev22G}   has  exactly $n-1$ roots $x^n_{i}$\,, $i=1, \ldots, n-1$, in the interval $(-1,1)$, and an additional isolated root $x_n:=x^n_{n}$  in the interval $(\frac{6+\sqrt{2}}{4},2]$. Furthermore,  the isolated root is $x_{n}=\cosh \omega_n$, where $\omega_n$ is the unique root of the equation 
\begin{equation}\label{greatexpressionh20}
\coth(n\omega)=\frac{3 \cosh\omega -4}{\sinh\omega}\,,\qquad   \omega\in (0,\infty)\,.
\end{equation}
Besides, $x_1=2$ and the sequence of isolated roots $(x_n)$ decreasingly converges to the limit value 
$x_\infty=({6+\sqrt{2}})/{4}$.
\end{proposition}
\begin{figure}[ht]
\begin{picture}(0,120)(0,0)%
\put(110,-10){\includegraphics[scale=0.37]{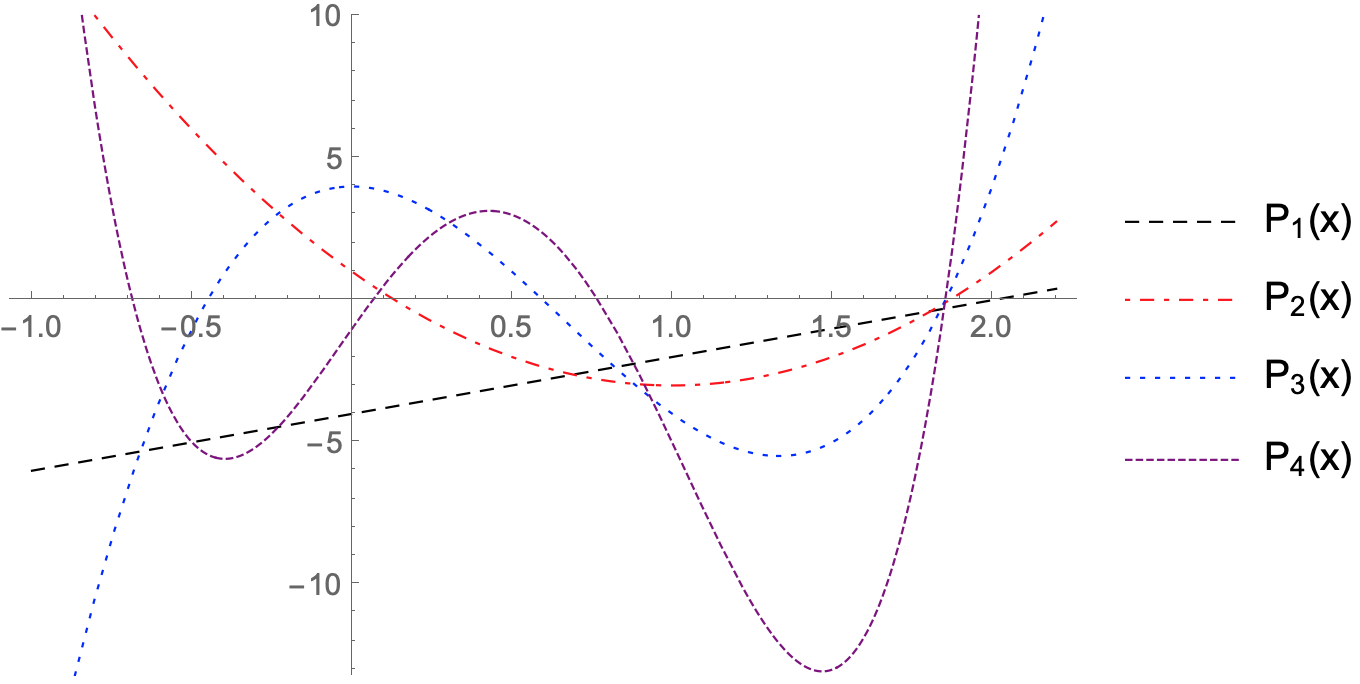}}
\end{picture}
\caption{Polynomials $P_n(x)\,,\ n=1,\ldots,4$.
Each polynomial has an isolated root in the interval $(\frac{6+\sqrt{2}}{4},2]$.}
\end{figure}

As it is shown in the proof (see Section \ref{subsec:rem}), for any value of $n$, the unique root $\omega_n$ of equation \eqref{greatexpressionh20} lies in the interval $(\omega_\infty,\omega_1]=(\log(2 + \sqrt{2}),\log(2 + \sqrt{3})]$. Consequently, the   isolated root 
 $x_{n}=\cosh \omega_n$ of polynomial $P_n(x)$ lies in the narrow interval 
 $(\frac{6+\sqrt{2}}{4},2]$. Having the root well located makes it  easy to approach it by any numerical method. In Table \ref{Tab:pos} we show some of these roots after 10 steps with bisection method.

\subsubsection*{Polynomials $C_n$}   
The polynomial $C_n$ of degree $n$ is defined as
\begin{equation} \label{eq:polCm}
C_n(x)=P_n(x)+U_n(x)\,,\qquad n\in \mathbb{N},
\end{equation}
where $P_n$ is the polynomial of degree $n$ defined above and 
$U_n$ is the Chebyshev polynomial of   second kind of degree $n$. 
Consequently, all the properties of $C_n$ are consequence of this definition, including its recursive definition
\begin{align} 
C_0(x)&=0\,,\nonumber\\
C_1(x)&=4(x-1)\,,\nonumber\\
C_n(x)&=2x\, C_{n-1}(x)-  C_{n-2}(x) \,.\label{chebyshev22R}
\end{align} 
Observe that the recursive formula 
\eqref{chebyshev22R} is the same as \eqref{chebyshev22G}  for $P_n$ and \eqref{chebyshev22} for $U_n$, with just different starting values $C_0$ and $C_1$. As in previous cases, it is worth writing $C_n(x)$ in terms of a determinant
\begin{equation} \label{chebyshev2det4}
C_n(x)= \left|\begin{array}{ccccc}
4(x-1) & 0 &  & & \\
  1& 2x & 1&& \\
   & \ddots &\ddots &\ddots  & \\
   &   &   1 & 2x&1  \\  
   &&& 1 & 2x
\end{array}\right|\,.
\end{equation}

\begin{proposition}\label{proposition:C(x)}
Each polynomial $C_n(x)$ has exactly $n-1$ roots $x^n_{i}$\,, $i=1, \ldots, n-1$, in the interval $(-1,1)$, and the additional isolated root $x_n=1$. 
\end{proposition}
\begin{proof}
It is straightforward if we use the determinant 
\eqref{chebyshev2det4}, where we get
\begin{equation}\label{CtoU}
C_n(x)=4(x-1)U_{n-1}(x)
\end{equation}
Consequently the roots of polynomial $C_n$  are 
 the isolated root $x_n=1$, and the  $n-1$ roots of the 
Chebyshev polynomial  of   second kind of degree ${n-1}$.
\end{proof}

 \begin{corollary}\label{Cpositive}
If $x>1$, then $C_n(x)>0\ \forall n\in\mathbb{N}$\,.
\end{corollary}
\begin{proof}   It is straightforward from the previous proposition.
 \end{proof}

In the following Lemma we give some technical  properties of polynomials $C_n(x)$ that we will need in Section \ref{sec_proofs}.

\begin{lemma}\label{lemma:C(x)} 
For the polynomial $C_n$ of degree $n$  defined in \eqref{eq:polCm} or \eqref{chebyshev22R} the following properties hold:
\begin{enumerate}
	\item\label{unoo} $  C_n(1)=0\,,  
	\forall n\in\mathbb{N}\,.$
         \item \label{doss} $C_n(x)=2U_{n-2}(x)
         -4U_{n-1}(x)+2U_{n}(x)\,,
	\forall n\in\mathbb{N}\,.$
	\item \label{tress} If $x>1$\,, then \  
	$0<C_n(x)\le C_{n+1}(x)\,,
	  \quad \forall n\in\mathbb{N}\,.$    
	 \item \label{cuatroo}   If $x>1$\,, then \  
	 $2C_{n}(x)\le C_{n-1}(x)+C_{n+1}(x)\,, 
	 \quad \forall n\in\mathbb{N}$\,.
\end{enumerate}	
\end{lemma}
\begin{proof} Part  \ref{unoo} is straightforward. 
Part \ref{doss} is also straightforward if we use 
 the  relationship~\eqref{PtoU}
and definition \eqref{eq:polCm}.

To prove part \ref{tress},  as $x>1$, we set
	$x=\cosh\omega$ in \eqref{CtoU}, 
	to obtain
\begin{equation}\label{greatexpressionRR}
C_n(\cosh\omega)=  
 \frac{4(\cosh\omega -1)\sinh (n\omega) }{\sinh \omega}>0 \,. 
\end{equation}
As 	${\sinh \omega}$ is an increasing function, it holds $\sinh (n\omega)<\sinh ((n+1)\omega)$, and consequently $C_n(x)\le C_{n+1}(x)$\,. 

Finally, to prove \ref{cuatroo}, as
	\[
	\sinh((n-1)\omega)+\sinh((n+1)\omega)=2\cosh\omega\sinh(n\omega)\ge 2\sinh(n\omega)\,,
	\]
we can use again \eqref{greatexpressionRR},  for $x>1$, to obtain that $2C_{n}(x)\le C_{n-1}(x)+C_{n+1}(x)$
	  $\forall n\in\mathbb{N}$.	
\end{proof}

\section{Proofs of theorems in Sections \ref{sec:main} and \ref{sec:preliminar}}\label{sec_proofs}

In this section we give the proofs of  the main results in the paper.

\subsection{Positivity of the Crank-Nicolson method}
Recall that the  Crank Nicolson method is positive if and only if all the elements in matrix $A_m$ \eqref{Aimpar}-\eqref{Apar} are positive.  The following lemma simplifies the proof of Theorem   \ref{Th:pos}. 

\begin{lemma}\label{lemma_Pm}
If $x>1$, then all the elements in matrix $A_m(x)$ \eqref{Aimpar}-\eqref{Apar} are non-negative if and only if polynomial $P_{m}(x)$ is non-negative.
\end{lemma}

\begin{proof}
Recall that polynomials $U_n(x)$ are positive for $x > 1$. 
Polynomials $C_n(x)$ are also
positive for $x>1$ (see Corollary \ref{Cpositive}). 
Consequently, all extra-diagonal elements are positive.

For the diagonal elements $A_{m}^{(i,i)}(x)$, the positivity of polynomials $U_n(x)$ and $C_n(x)$ for all $n\in\mathbb{N}$ implies that 
\[
\min_i A_{m}^{(i,i)}(x)=A_{m}^{(1,1)}(x)=A_{m}^{(m,m)}(x)=\frac{P_{m}(x)}{U_{m}(x)}.
\]
Then the analysis of the positivity of the elements in  $A_m(x)$ is reduced to the
positivity of polynomial $P_{m}(x)$.
\end{proof}

\subsubsection*{Proof of Theorem \ref{Th:pos}}
\begin{enumerate}  
\item 
From Lemma \ref{lemma_Pm}, we just have to study the positivity of polynomials $P_m(x)$ for $x>1$.  From Proposition \ref{proposition:p(x)}, $P_m(x)$ has  exactly $m-1$ roots in the interval $(-1,1)$ and an additional isolated real root 
$x_m^{(p)}=\cosh(\omega_m^{(p)})$ in the interval $(\frac{6+\sqrt{2}}{4},2]$, where $\omega_m^{(p)}$ is the unique root of equation \eqref{greatexpressionh20}, that is the same as \eqref{wroot_pos}. As  $x=1+1/s$, then 
the CFL coefficient is  $s=1/(x-1)$, and  inequality \eqref{sCFL_pos} is obtained. 

\item 
In the limit case, Proposition \ref{proposition:p(x)} gives us the limit value 
$x_\infty^{(p)}=({6+\sqrt{2}})/{4}$. Then  for the CFL coefficient we get the bound $s_\infty^{(p)}=1/(x_\infty^{(p)}-1)=2(2-\sqrt{2})$, and  inequality~\eqref{cota_positividad} is obtained. \qed
\end{enumerate}

\subsection{Contractivity of the Crank-Nicolson method.}  
Before writing  the proof of Theorem \ref{Th:con}, we need  two technical lemmas. In Lemma \ref{proposition_norm_inf} below, we compute the maximum norm of matrix $A_{m}(x)$  in terms of polynomials  $U_n(x)$, $P_n(x)$ and   $C_n(x)$.  
Then, in  Lemma~\ref{theorem_norm_inf} we get 
the inequalities needed for the contractivity condition $\norm{A_m}_\infty\le 1$. In contrast to positivity, in the analysis of the contractivity it is necessary to distinguish between the even and odd cases.

\begin{lemma}\label{proposition_norm_inf}\mbox{}
\begin{enumerate}	
	\item If $m\ge 3$ is a natural odd number, then the 
	maximum norm of Crank-Nicolson matrix~is 
\begin{equation}\label{maxnormodd}
	\norm{A_m}_\infty =\tfrac{1}{U_m}\left( \abs{P_m+\sum_{n=1}^{\frac{m-1}{2}}C_{2n-1}}+ 2\sum_{i=0}^{\frac{m-3}{2}} \sum_{n=1}^{\frac{m-(2 i+1)}{2}} C_{2 n+i}   \right)\,.
\end{equation}
	\item If $m\ge 2$ is a natural even number, 
	 then the norm of Crank-Nicolson matrix~is \begin{equation}\label{maxnormeven}
\norm{A_m}_\infty = \tfrac{1}{U_m}\left( \abs{P_m+\sum_{n=1}^{\frac{m}{2}-1}C_{2n}}+ \sum_{i=1}^{\frac{m}{2}-1} \sum_{n=1+i}^{m-i} {C}_n+\sum_{n=1}^{m / 2} {C}_{2n-1}   \right)\,.
\end{equation}
\end{enumerate}	 
\end{lemma}

\begin{proof} 
Note that the norm in \eqref{maxnormodd} is obtained from the sum of the elements in the central row of $A_m$ (odd case), while the norm in 
\eqref{maxnormeven}  is obtained from the sum of the elements in any of the two symmetric central 
rows (even case). To get this result we will proof the following inequalities
\[
\hspace{3cm}
\sum_{j=1}^m \abs{A_m^{(i,j)}}\le \sum_{j=1}^m \abs{A_m^{(i+1,j)}}\,,\quad i=1,\ldots,\tfrac{m-1}{2}\text{(odd case)}\,   \tfrac{m-2}{2}\text{(even case)}.
\]
For the sake of simplicity we will denote $A_m^i$ to the sum
$\sum_{j=1}^m \abs{A_m^{(i,j)}}{U_m}$. 
Observe that the symmetry of matrices (\ref{Aimpar}-\ref{Apar}) makes $A_m^i= A_m^{m+1-i}$\,, $i=1\,,\ldots\,,(m-1)/2$ (odd case) or   $i=1\,,\ldots\,,m/2$ (even case).
Then, we will proof that
\begin{align*}
&A_m^1\le A_m^2\le \cdots \le A_m^{\frac{m-1}{2}}\le A_m^{\frac{m+1}{2}}\qquad \hspace{1.1cm}\text{(odd case)}\\
&A_m^1\le A_m^2\le \cdots \le A_m^{\frac{m-2}{2}}\le A_m^{\frac{m}{2}}=A_m^{\frac{m+2}{2}}\qquad \text{(even case)}
\end{align*} 
Consequently, $\norm{A_m}_\infty$  is obtained from the sum of the elements in the central row (odd case) or  from the sum of the elements in any of the two  central rows (even case).

Recall that for a given matrix $A$ with positive  extra-diagonal elements the following equality trivially  holds
\begin{equation}\label{maxnorm}
\norm{A}_\infty =\max_i\left(\sum_{j=1}^m \abs{a_{ij}}\right)=\max_i\left(\abs{a_{ii}}+\sum_{j\ne i} a_{ij}\right)\,.
\end{equation}	
In our case   $C_n(x)\ge 0$ for  $x\ge 1$ (see Proposition \ref{proposition:C(x)}). Consequently all extra-diagonal elements of matrices  
	\eqref{Aimpar} and \eqref{Apar}  are non-negative and we can use \eqref{maxnorm} to compute~$\norm{A_m}_\infty$.

For the  first and the second row of matrix  \eqref{Aimpar} or \eqref{Apar}  we get the sums  
\[
A_m^1=\abs{P_m}+\sum_{i=1}^{m-1}{C_i}\,,\qquad 
A_m^2=\abs{P_m+C_{m-2}}+\sum_1^{m-1}{C_i}+\sum_2^{m-1}{C_i}-{C_{m-2}}\,.
\]  
Then, after cancelling terms, the difference $A_m^1-A_m^2$ is  
\begin{align*}
A_m^1-A_m^2&=
\abs{P_m}-\abs{P_m+C_{m-2}}-\sum_2^{m-1}{C_i}+
{C_{m-2}}\,.\\
\intertext{Adding and subtracting $C_{m-2}$ in the term $\abs{P_m}$, we can write}
A_m^1-A_m^2&\le \abs{P_m+C_{m-2}}+C_{m-2}-\abs{P_m+C_{m-2}}-\sum_2^{m-1}{C_i}+
{C_{m-2}}\\
&=2{C_{m-2}}-\sum_2^{m-1}{C_i}\le
 {C_{m-3}}+{C_{m-1}}-\sum_2^{m-1}{C_i}=
  -\sum_2^{m-4}{C_i}-{C_{m-2}}\le 0\,,
\end{align*}  
{where we have used  $2{C_{m-2}} \le
 {C_{m-3}}+{C_{m-1}}$ from item  \eqref{cuatroo} in Lemma \ref{lemma:C(x)}. Observe that property~\eqref{tress} in Lemma \ref{lemma:C(x)}    allows us to finally write} 
\begin{equation}\label{eq:A1A2dif}
A_m^1-A_m^2  \le 
  -\sum_2^{m-3}{C_i}\le 0\,.
\end{equation}
If we proceed in the same way for the 
difference $A_m^2-A_m^3$, after cancelling terms, and after adding and subtracting $C_{m-4}$ in the term $\abs{P_m+C_{m-2}}$, we get  the following inequality   
\begin{align*}
A_m^2-A_m^3&\le 2C_{m-4}-\sum_3^{m-2}{C_i} \,.\\
\intertext{Again, the use of  properties \eqref{cuatroo} and  \eqref{tress}, in this order,    from Lemma \ref{lemma:C(x)} makes it possible to write an inequality analogous to \eqref{eq:A1A2dif}}
A_m^2-A_m^3& \le  C_{m-5}+C_{m-3}
  -\sum_3^{m-2}{C_i}\le   -\sum_3^{m-4}{C_i}\le 0\,.
\end{align*}

The proof follows in the same way for the odd and even case up to the last step when we achieve the central  row (odd case) or  the two central rows (even case). Then we have to consider two different cases:
\begin{enumerate}	
	\item If $m$ is an odd number, the last step consists in studying the difference 
	$A_m^{\frac{m-1}{2}}-A_n^\frac{m+1}{2}$, 
	where $A_n^\frac{m+1}{2}$ represents the sum of the elements in the central row. After cancelling terms,
we can write
\begin{align*}
A_m^{\frac{m-1}{2}}-A_m^\frac{m+1}{2}&=\abs{P_m+\sum_{n=2}^{\frac{m-1}{2}}C_{2n-1}}+C_1-\abs{P_m+\sum_{n=1}^{\frac{m-1}{2}}C_{2n-1}}-C_{\frac{m+1}{2}}\\
\intertext{Adding and subtracting $C_{1}$ in the term $\abs{P_m+\sum_{n=2}^{\frac{m-1}{2}}C_{2n-1}}$, we get}
A_m^{\frac{m-1}{2}}-A_m^\frac{m+1}{2}&  \le \abs{P_m+\sum_{n=1}^{\frac{m-1}{2}}C_{2n-1}}+2C_1-\abs{P_m+\sum_{n=1}^{\frac{m-1}{2}}C_{2n-1}}-
C_{\frac{m+1}{2}}\\
&=2C_1-C_{\frac{m+1}{2}}\le  C_0+C_2-C_{\frac{m+1}{2}}=C_2-C_{\frac{m+1}{2}}\le 0\,,
\end{align*}
where we have used  $2{C_{1}} \le
 {C_{0}}+{C_{2}}$ from  property \eqref{cuatroo} in Lemma \ref{lemma:C(x)}.

Consequently, if     $m\ge 3$ is an odd number, the maximum value of $A_m^i$ is obtained in the central row ${A_m^\frac{m+1}{2}}$ and we can conclude 
\[
\norm{A_m}_\infty =
\frac{A_m^\frac{m+1}{2}}{U_m}=
 \frac{1}{U_m}\left( \abs{P_m+\sum_{n=1}^{\frac{m-1}{2}}C_{2n-1}}+ 2\sum_{i=0}^{\frac{m-3}{2}} \sum_{n=1}^{\frac{m-(2 i+1)}{2}} C_{2 n+i}   \right)\,.
\]  
 
\item If $m$ is an even number, then there is not a central row but two central symmetric rows 
	$A_m^\frac{m}{2}$ and $A_m^\frac{m+2}{2}$, and the maximum value is obtained at any of these two files. Now, in the last step of the proof,  if  $m\ge 4$, we have to write  the difference \mbox{$A_m^{\frac{m-2}{2}}-A_m^\frac{m}{2}$}. Note that for the simple case $m=2$, it holds $A_m^{1}=A_m^{2}$. After cancelling terms, we can write
\begin{align*}
A_m^{\frac{m-2}{2}}-A_m^\frac{m}{2}&=
\abs{P_m+\sum_{n=2}^{\frac{m}{2}-1}C_{2n}}+C_2   -\abs{P_m+\sum_{n=1}^{\frac{m}{2}-1}C_{2n}}-C_{\frac{m}{2}}-C_{\frac{m+2}{2}}  \\
\intertext{Adding and subtracting $C_{2}$ in the term $\abs{P_m+\sum_{n=2}^{\frac{m}{2}-1}C_{2n}}$, we get}
A_m^{\frac{m-2}{2}}-A_m^\frac{m}{2} &\le \abs{P_m+\sum_{n=1}^{\frac{m}{2}-1}C_{2n}}+2C_2   -\abs{P_m+\sum_{n=1}^{\frac{m}{2}-1}C_{2n}}-C_{\frac{m}{2}}-C_{\frac{m+2}{2}}\\
&=2C_2-C_{\frac{m}{2}}-C_{\frac{m+2}{2}}\le  C_1+C_3-C_{\frac{m}{2}}-C_{\frac{m+2}{2}} \le 0\,,
\end{align*}
where we have used  $2{C_{2}} \le
 {C_{1}}+{C_{3}}$ from  property \eqref{cuatroo} in Lemma \ref{lemma:C(x)}.

Consequently, if     $m$ is an even number, the maximum value of $A_m^i$ is obtained in the  row ${A_m^\frac{m}{2}}$ and we can conclude 
\[
\mbox{}\!\!\!\!\norm{A_m}_\infty =\max_i\left(\abs{a_{ii}}+\sum_{j\ne i} a_{ij}\right)= \tfrac{1}{U_m}\left( \abs{P_m+\!\!\sum_{n=1}^{\frac{m}{2}-1}\!\!C_{2n}}+\!\! \sum_{i=1}^{\frac{m}{2}-1} \sum_{n=1+i}^{m-i} \!\!{C}_n+\!\!\sum_{n=1}^{m / 2}\!\! {C}_{2n-1}   \right)\!\!\qed
\] 
\end{enumerate}	
\renewcommand{\qedsymbol}{}
\end{proof}
Once we have got  the maximum norm 
of  matrix~${A_m}$
in terms of polynomials  $U_m$, $P_m$ and $C_m$, we can get bounds $s_m^{(c)}$ for contractivity for any value of $m$ if we are able to solve the corresponding inequality 
$\norm{A_m}_\infty\le 1$. This is done in the following lemma. In Figure~\ref{Fig:normaAm} we have plot  
 $\norm{A_m(s)}_\infty$ for some values of $m$, and we have also added some contractivity bounds 
$s_m^{(c)}$.  
\begin{figure}[h!] 
\begin{picture}(0,180)(0,-10)
\put(55,-12){\includegraphics[scale=0.508]{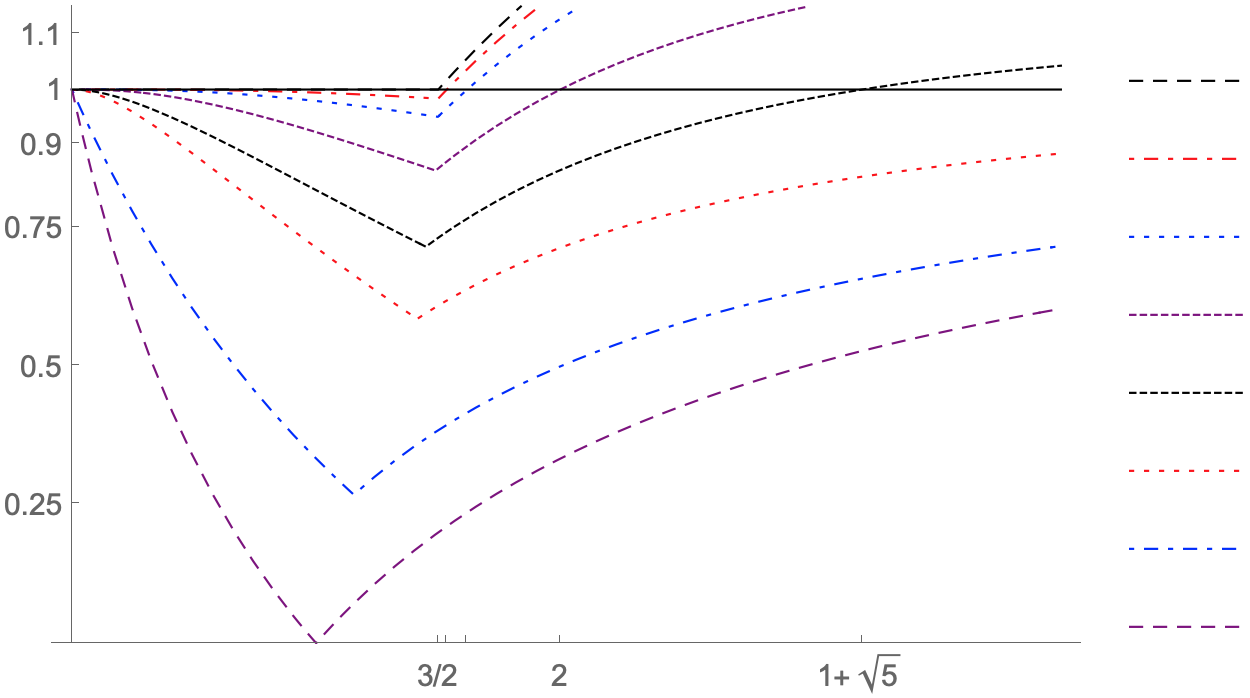}}
\put(274,10){\footnotesize${s_4^{(c)}}$}
\put(196,10){\footnotesize${s_5^{(c)}}$}
\put(173,10){\footnotesize${s_7^{(c)}}$}
\put(380,05){\footnotesize$\norm{A_1(s)}$}
\put(380,25){\footnotesize$\norm{A_2(s)}$}
\put(380,45){\footnotesize$\norm{A_3(s)}$}
\put(380,65){\footnotesize$\norm{A_4(s)}$}
\put(380,85){\footnotesize$\norm{A_5(s)}$}
\put(380,105){\footnotesize$\norm{A_7(s)}$}
\put(380,125){\footnotesize$\norm{A_9(s)}$}
\put(380,145){\footnotesize$\norm{A_{21}(s)}$}
\end{picture}
\caption{$\norm{A_m(s)}$ for different values of 
$m\in\{1,2,3,4,5,7,9,21\}$. 
When $m\in\{1,2,3\}$ it holds 
$\norm{A_m(s)}<1$, for all $s>0$. 
For $m\geq 4$, $\norm{A_m(s)}$ cuts the line $s=1$ 
at~$s_m^{(c)}$. The sequence $(s_m^{(c)})$  is  strictly monotonically decreasing  with all the terms 
in the   interval~$\big(3/2, 1+\sqrt{5}\,\big]$.}\label{Fig:normaAm}
\end{figure}
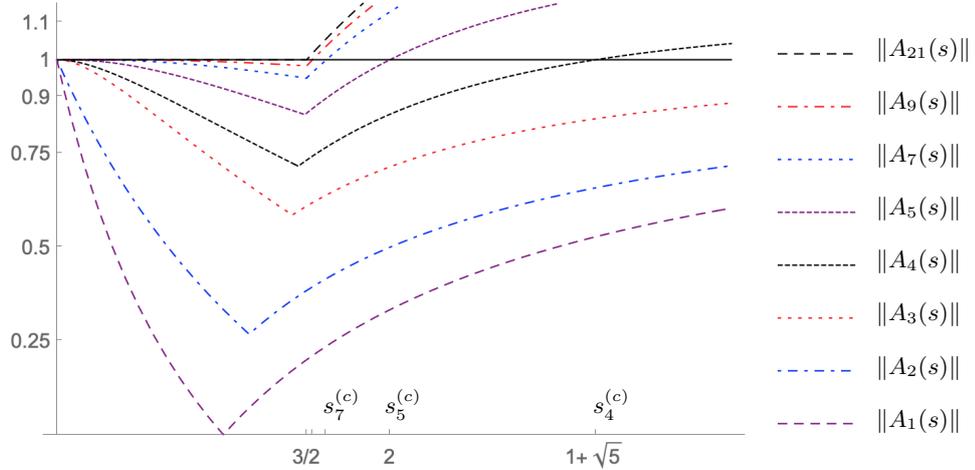 
\begin{lemma}\label{theorem_norm_inf} \mbox{}
\begin{enumerate}
     \item (Ood case) If $m$ is a natural odd number, then for the maximum norm of Crank-Nicolson matrix we have
 \[
\norm{A_m}_\infty\le 1  \  \iff \  
\frac{2\sinh\frac{(m-1)\omega}{4}\sinh\frac{(m+1)\omega}{4}}{\sinh\frac{(m+1)\omega}{2}}\le \sinh\frac{\omega}{2}
\]
where $\omega=\operatorname{arccosh}(1+1/s)$. 
In the limit, when $m\to\infty$, we get contractivity if and only if
\begin{equation}\label{limitodd}
e^{-\omega/2}\le \sinh\frac{\omega}{2}
\end{equation}
\item (Even case)
If $m$ is a natural even number, then  we have
\[
\norm{A_m}_\infty\le 1  \  \iff \  
\frac{\sinh\frac{m\omega}{2}\left(\sinh\frac{(m+2)\omega}{2}-\sinh\frac{m\omega}{2}\right)}{\sinh\frac{(m+1)\omega}{2}\sinh\frac{m\omega}{4}\sinh\frac{(m-2)\omega}{4} } 
\ge \frac{\sinh \omega}{\sinh^2\frac{\omega}{2}}
\]
In the limit, when $m\to\infty$, we get contractivity if and only if 
\begin{equation}\label{limiteven}
2\left(-1+e^{\omega}\right)\ge \frac{\sinh \omega}{\sinh^2\frac{\omega}{2}}
\end{equation}
\end{enumerate}
If $\omega>0$, inequalities \eqref{limitodd} and \eqref{limiteven} are equivalent, and they are true iff $\omega\ge \log 3$. In the variable $s$ this is equivalent to the known restriction \  $s\le 3/2$.
\end{lemma}
\begin{proof}\mbox{}
\begin{enumerate}
	\item Remember that, in the odd case, the norm in \eqref{maxnormodd} is obtained from the sum of the elements in the central row of $A_m$. 
	In that case,   the diagonal element in the central row   ${P_m+\sum_{n=1}^{\frac{m-1}{2}}C_{2n-1}}$  
can be written in closed form  as
\[
{P_m+\sum_{n=1}^{\frac{m-1}{2}}C_{2n-1}}=
{-U_m+\sum_{n=1}^{\frac{m+1}{2}}C_{2n-1}}=
-\frac{\sinh((m+1)\omega)}{\sinh\omega}+\frac{4(x-1)\sinh^2\frac{(m+1)\omega}{2}  }{\sinh^2\omega} \,,
\]
where we have changed $P_m=C_m-U_m$ and we have considered the angle variable
$\omega=\operatorname{arccosh}x$. 
For any value of $m$, 
the unique positive root $\omega_m$  of this diagonal element   lies in the interval 
$(\log 3,\log(2 + \sqrt{3})]$. This root can easily be obtained from the simplified equation in the variable 
$s=(\cosh\omega -1)^{-1}$
\[
{\sqrt{1+2s}}=2\tanh \frac{(1+m)\operatorname{arccosh}(1+1/s)}{2}  \,.
\]
In this variable, the unique positive root $s_m$   
lies in the interval $[1,3/2)$. Observe that the sequence of roots $(s_m)$ increasingly converges to the limit value \mbox{$s_{\infty}=3/2$}.  
When $s\in(0,s_m)$ the diagonal element
${P_m+\sum_{n=1}^{\frac{m-1}{2}}C_{2n-1}}$ is positive, and, from~\eqref{maxnormodd},    we easily obtain
$\norm{A_m}_\infty< 1$.

On the other hand, when \mbox{$s\in[s_m,\infty)$}, we have
${P_m+\sum_{n=1}^{\frac{m-1}{2}}\!\!C_{2n-1}}\!\le 0$\,, and,
from~\eqref{maxnormodd}, the inequality for contractivity is
\[
\norm{A_m}_\infty =  \tfrac{1}{U_m}
\left( {-P_m-\sum_{n=1}^{\frac{m-1}{2}}C_{2n-1}}+ 
2\sum_{i=0}^{\frac{m-3}{2}} \sum_{n=1}^{\frac{m-(2 i+1)}{2}} C_{2 n+i}  \right)\le 1\,.
\]
This is equivalent to 	 
\[
 {-\sum_{n=1}^{\frac{m-1}{2}}C_{2n-1}}+ 
2\sum_{i=0}^{\frac{m-3}{2}} \sum_{n=1}^{\frac{m-(2 i+1)}{2}} C_{2 n+i}  \le C_m\,,
\]
or
\[
 {\sum_{n=1}^{\frac{m+1}{2}}C_{2n-1}}- 
2\sum_{i=0}^{\frac{m-3}{2}} \sum_{n=1}^{\frac{m-(2 i+1)}{2}} C_{2 n+i}  \ge 0\,.
\]
As $x>1$, in the angle variable 
$\omega=\operatorname{arccosh} x>0$, we 
can write $C_m(\cosh\omega)=4(x-1)\sinh(m\omega)/\sinh\omega$. Consequently, the previous inequality is reduced to
\[
 {\sum_{n=1}^{\frac{m+1}{2}}\sinh((2n-1)\omega)}- 
2\sum_{i=0}^{\frac{m-3}{2}} \sum_{n=1}^{\frac{m-(2 i+1)}{2}} \sinh((2n+i)\omega)  \ge 0\,.
\]
If we use the closed formulas
for the expansions
$\sum_k \sinh (k\omega)$, $\sum_k \sinh (2k\omega)$, and $\sum_k \sinh ((2k-1)\omega)$, then we can write the previous inequality as
\[
\frac{4(x -1)\sinh\frac{(m+1)\omega}{2} }{\sinh^2 \omega}\left(\sinh\frac{(m+1)\omega}{2}-2\frac{\sinh\frac{(m-1)\omega}{4}\sinh\frac{(m+1)\omega}{4} }{\sinh\frac{\omega}{2}}\right)\ge 0\,.
\]
And, for $x>1$, this is true if and only if
\begin{equation}\label{ineqodd}
\frac{2\sinh\frac{(m-1)\omega}{4}\sinh\frac{(m+1)\omega}{4}}{\sinh\frac{(m+1)\omega}{2}}\le \sinh\frac{\omega}{2}\,.
\end{equation}
In this way, for any value of $m$, we get contractivity if and only if  $\omega\ge \omega_m^{(c)}$, where $\omega_m^{(c)}$ is the unique positive root of the corresponding equality equation. Finally, going back to variable $s$, we get contractivity if and only if $s\le s_m^{(c)}:= 1/(\cosh \omega_m^{(c)} -1)$.

Computing the limit in \eqref{ineqodd}, when $m\to\infty$, we get contractivity if and only if
\[
e^{-\omega/2}\le \sinh\frac{\omega}{2}\,.
\]
And this is true if and only if $\omega\ge \log 3$. This is
$x=\cosh\omega\ge 5/3$ or $s\le 3/2$.
	\item {In the even case, the norm in \eqref{maxnormeven} is obtained from the sum of the elements in any of the  two central symmetric rows.   
	The diagonal element  ${P_m+\sum_{n=1}^{\frac{m}{2}-1}C_{2n}}$  
	in any of this central rows 
can be written in closed form  as 
\[
P_m+\!\sum_{n=1}^{\frac{m}{2}-1}C_{2n}=
{-U_m+\sum_{n=1}^{\frac{m}{2}}C_{2n}}=
-\frac{\sinh((m+1)\omega)}{\sinh\omega}+\frac{4(x-1)\sinh\frac{m\omega}{2}\sinh\frac{(m+2)\omega}{2}  }{\sinh^2\omega} \,,
\]
where $\omega=\operatorname{arccosh}x$ and 
$x=1+1/s$. Again, the unique positive root $s_m$  of this central diagonal element 
lies in the interval $[1,3/2)$. For any value of $m$, this root can easily be obtained from the simplified equation 
\[
{\sqrt{1+2s}}=4\,\frac{\sinh\frac{m\omega}{2}\sinh\frac{(m+2)\omega}{2}}{\sinh((m+1)\omega)}  \,.
\]
Observe that this sequence of roots $(s_m)_m$ increasingly converges to the limit value $s_{\infty}=3/2$.  When $s\in(0,s_m)$ the diagonal element
${P_m+\sum_{n=1}^{\frac{m}{2}-1}C_{2n}}$ is positive, 
and, from~\eqref{maxnormeven}, we easily obtain
$\norm{A_m}_\infty< 1$.  On the other hand, when \mbox{$s\in[s_m,\infty)$}, 
the diagonal element  is negative, and,
from~\eqref{maxnormeven}, 
the inequality for contractivity is
\[
\norm{A_m}_\infty =  \tfrac{1}{U_m}
\left(  {-P_m-\sum_{n=1}^{\frac{m}{2}-1}C_{2n}}+ \sum_{i=1}^{\frac{m}{2}-1} \sum_{n=1+i}^{m-i} {C}_n+\sum_{n=1}^{m / 2} {C}_{2n-1}  \right)\le 1\,.
\]}
This is equivalent to
\[
  {-\sum_{n=1}^{\frac{m}{2}-1}C_{2n}}+ \sum_{i=1}^{\frac{m}{2}-1} \sum_{n=1+i}^{m-i} {C}_n+\sum_{n=1}^{m / 2} {C}_{2n-1}  \le C_m\,,
\]
or
\[
  {\sum_{n=1}^{\frac{m}{2}}C_{2n}}- \sum_{i=1}^{\frac{m}{2}-1} \sum_{n=1+i}^{m-i} {C}_n-\sum_{n=1}^{m / 2} {C}_{2n-1}  \ge 0\,.
\]
As  in the odd case, now we can use  the angle variable 
$\omega$ to reduce the previous inequality.
\[
  {\sum_{n=1}^{\frac{m}{2}}\sinh(2n\omega)}- \sum_{i=1}^{\frac{m}{2}-1} \sum_{n=1+i}^{m-i} \sinh(n\omega)-\sum_{n=1}^{m / 2} \sinh((2n-1)\omega)  \ge 0\,.
\]
Finally, the 
closed formulas    for the expansions
$\sum_k \sinh (k\omega)$, $\sum_k \sinh (2k\omega)$, and $\sum_k \sinh ((2k-1)\omega)$, allow us to reduce the inequality to
\[
\tfrac{4(x -1)}{\sinh \omega}\left(\frac{\sinh\frac{m\omega}{2}}{\sinh \omega} \left(\sinh\tfrac{(m+2)\omega}{2}-\sinh\tfrac{m\omega}{2}\right)
-\frac{\sinh\frac{(m+1)\omega}{2}\sinh\frac{m\omega}{4}\sinh\frac{(m-2)\omega}{4} }{\sinh^2\frac{\omega}{2}}\right)\ge 0\,.
\]
And, for $x>1$, this is true if and only if
\begin{equation}\label{ineqeven}
 \frac{\sinh\frac{m\omega}{2}\left(\sinh\frac{(m+2)\omega}{2}-\sinh\frac{m\omega}{2}\right)}{\sinh\frac{(m+1)\omega}{2}\sinh\frac{m\omega}{4}\sinh\frac{(m-2)\omega}{4} } 
\ge \frac{\sinh \omega}{\sinh^2\frac{\omega}{2}}\,.
\end{equation}
In this way, for any value of $m$, we get contractivity if and only if  $\omega\ge \omega_m^{(c)}$, where 
$\omega_m^{(c)}$ is the unique positive root of the corresponding equality equation 
in \eqref{ineqeven}. Finally, going back to variable $s$, we get contractivity if and only if $s\le 1/(\cosh \omega_m^{(c)} -1)$.

Computing the limit in \eqref{ineqeven}, when $m\to\infty$, we get contractivity if and only if 
\[
2\left(-1+e^{\omega}\right)\ge \frac{\sinh \omega}{\sinh^2\frac{\omega}{2}}\,.
\]
And this is true if and only if $\omega\ge \log 3$. This is
$x=\cosh\omega\ge 5/3$ or $s\le 3/2$.\qedhere
\end{enumerate}
\end{proof}

\subsubsection*{Proof of Theorem \ref{Th:con}}
\begin{enumerate}
\item For $m\in\{1,2,3\}$ the proof is straightforward.
\item For $m\in \mathbb{N}$,  $m>3$\,,  the proof is straightforward
from the previous lemma. 
The inequality for contractivity in the variable  
$s=1/({\cosh \omega -1})$ is
\[
s\le  s_m^{(c)}:=\frac{1}{\cosh \omega_m^{(c)} -1}
\]
where $\omega_m^{(c)}$ is the unique positive root of the   equation from \eqref{ineqodd}, if $m$ is odd, or from  \eqref{ineqeven}  if $m$ is even\,.
\item  It is also straightforward
from the previous lemma. In the limit, when $m\to\infty$, 
we get inequality \eqref{limitodd} from \eqref{ineqodd}, and 
inequality \eqref{limiteven} from \eqref{ineqeven}. These two inequalities \eqref{limitodd} and \eqref{limiteven} are equivalent 
if $\omega >0$, and they are true if and only if   $\omega\ge \log 3$. In the variable~$s$ this is equivalent to the known restriction \  $s\le 3/2$.\qed
\end{enumerate}

\subsection{Proof of the remaining results}
\label{subsec:rem}
\subsubsection*{Proof of Proposition \ref{proposition:p(x)}}
We divide the proof into two parts, the case   $x\in(-1,1)$ and the case   $x>1$. 
\begin{enumerate}
\item If  ${x\in(-1,1)}$, then the angular variable $\omega=\arccos x$, $\omega\in (0,\pi)$, and equality 
\eqref{PtoU},  allow us to write the polynomial $P_m$ in closed form as
\begin{align} 
P_m(\cos \omega)&=\frac{2\sin ((m-1) \omega)-4\sin (m\omega)+\sin ((m+1)\omega)}{\sin \omega}  \label{Pmdef}   \\
\intertext{If we convert  all the angles  in the numerator  
to the angle $m\omega$, we can write}
& = \frac{(3 \cos\omega -4)\sin (m\omega)}{\sin \omega}-   \cos (m\omega) \,,\label{greatexpression}
\end{align}
Then, from \eqref{greatexpression}, we get that the roots of $P_m(\cos \omega)$ in $(0,\pi)$ are 
the roots of 
the equation
\begin{equation}\label{rootstheta}
\tan(m\omega)=\frac{\sin\omega}{3 \cos\omega -4}\,,\quad 
  \omega\in (0,\pi)\,.
\end{equation}
The function on the right hand side, $f(\omega):=\sin\omega/(3 \cos\omega -4)$,
 is continuous and bounded in the interval $[0,\pi]$. It is decreasing in the interval  
 $(0,2 \arctan(1/\sqrt{7}))$ and increasing in the interval 
  $(2 \arctan(1/\sqrt{7}),\pi)$. Its maximum  value 
  \mbox{$f(0)=f(\pi)=0$} is obtained in the boundary, while the minimum value
  is $f(2 \arctan(1/\sqrt{7}))=-1/\sqrt{7}$~(see Figure \ref{fig_roots}).
\begin{figure}[ht] 
\begin{picture}(0,110)(0,0)%
\put(100,-5){\includegraphics[scale=0.28]{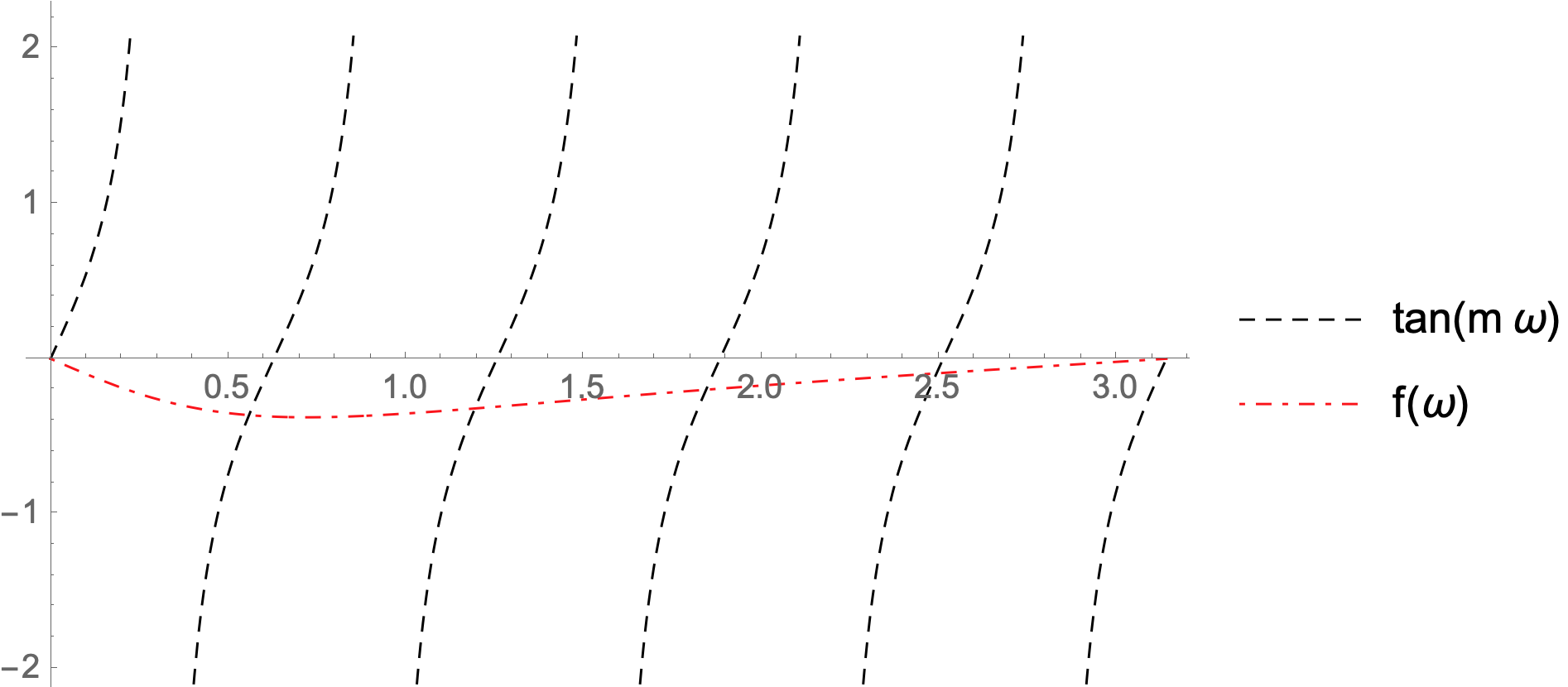}}%
\end{picture} 
\caption{Roots $0<\omega^m_{m-1}< \omega^m_{m-2}<\cdots <\omega^m_{1}<\pi\,,$ of equation \eqref{rootstheta}}
\label{fig_roots}
\end{figure}%
Thus, we can affirm that $\tan(m\omega)$ \ meets $m-1$ times the function $f(\omega)$,  and consequently there are $m-1$ roots $0<\omega^m_{m-1}< \omega^m_{m-2}<\cdots <\omega^m_{1}<\pi\,,$  in the interval~$(0,\pi)$. 
Now, going back to the variable $x=\cos\omega$, we can affirm that the polynomial $P_m(x)$ has $m-1$ roots
$x^m_{i}=\cos\omega^m_{i}$
 in the interval $(-1,1)$.
As the cosine function  is decreasing in the interval $(0,\pi)$, we can write the $m-1$ roots of $P_m(x)$ as 
\begin{equation*} 
 -1<x^m_{1}<x^m_{2}<\cdots <x^m_{m-1}<1\,.
\end{equation*}
\item  If $x>1$,  we consider the variable 
$\omega=\arccosh x$, $\omega\in (0,\infty)$. Again, with the help of equality \eqref{PtoU}, we can write the polynomial $P_m$ in closed form as
\begin{align} 
P_m(\cosh (\omega))&=\frac{2\sinh ((m-1) \omega)-4\sinh (m\omega)+\sinh ((m+1)\omega)}{\sinh \omega}\label{Pmdef2}\\
& = \frac{(3 \cosh\omega -4)\sinh (m\omega)}{\sinh \omega}-   \cosh (m\omega)\,,\label{greatexpressionh}
\end{align} 
where, as in the previous case, we have rewritten  the numerator 
in the terms of the angle $m\omega$.  
Then, from \eqref{greatexpressionh}, 
the roots of $P_m(\cosh (\omega))$ in $(0,\infty)$ are 
the roots  of the equation
\begin{equation}\label{greatexpressionh2}
\coth(m\omega)=\frac{3 \cosh\omega -4}{\sinh\omega}\,,\qquad   \omega\in (0,\infty)\,.
\end{equation}
The function on the right, $g(\omega):=(3 \cosh\omega -4)/\sinh\omega$, 
 is continuous and increasing in the interval $(0,\infty)$ 
 to the limit value of 3.  
 On the left, for any value of $m$, the function  
 $\coth(m\omega)$ is continuous and decreasing to the limit value of 1 (see Figure~\ref{fig_roots2}). 
  Thus, we can affirm that, for any value of $m$, equation \eqref{greatexpressionh2} has a unique root $\omega_{m}$  in the interval $(0,\infty)$\,. 
Now, going back to the variable $x=\cosh\omega$, we can affirm that the polynomial $P_m(x)$ has a unique root
$x_{m}=\cosh\omega_{m}$
 in the interval $(1,\infty)$.   
\begin{figure}[h!]
\begin{picture}(0,130)(0,0)%
\put(100,-5){\includegraphics[scale=0.38]{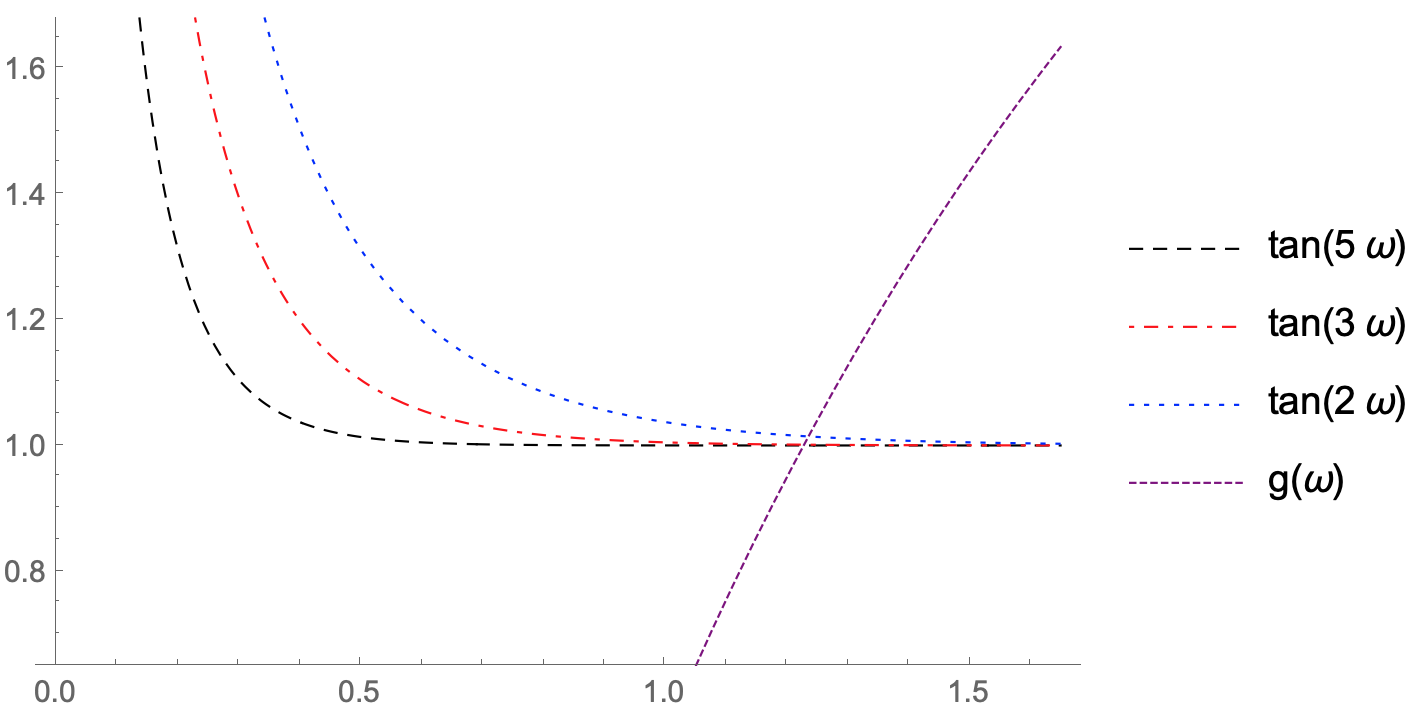}}%
\end{picture}   
\caption{Left hand side and right hand side of   equation \eqref{greatexpressionh2}}
\label{fig_roots2}
\end{figure}

In the limit, when $m$ tends to infinity, for any value of  
$\omega>0$, we have $\coth (m\omega){\searrow} 1$. Consequently, from \eqref{greatexpressionh2}, the sequence of  roots $(\omega_m)$ decreasingly converges to the limit value $\omega_\infty=\log(2 + \sqrt{2})$, this is the positive  solution of the limit  equation  $3 \cosh\omega -4=\sinh\omega$. Now, in the variable $x=\cosh\omega$, we can affirm that each polynomial $P_m(x)$ has a positive root
$x_m=\cosh\omega_m$
 in the interval $(1,\infty)$.  
 As the cosh function  is increasing in the interval $(0,\infty)$,   the sequence of  roots $(x_m)$ decreasingly converges to the limit value $x_\infty=(6 + \sqrt{2})/4$.
\[
x_\infty=\cosh\omega_\infty=\cosh\log(2 + \sqrt{2})=(6 + \sqrt{2})/4\approx 1.85355\,.    
\]  
\end{enumerate}
\vspace{-0.95cm}
\qed

\nocite{liu2007unified,
wang2017root,szaboqualitative,qi2019some}

\end{document}